\documentclass[amsart,11pt]{amsart}
\usepackage[utf8]{inputenc}
\usepackage[T1]{fontenc}
\usepackage{amsmath}
\usepackage{amssymb}
\usepackage{graphicx}
\usepackage[counterclockwise]{rotating}
\usepackage{mathtools}
\usepackage{wasysym}
\usepackage{enumerate}
\usepackage{amsfonts}
\usepackage{dirtytalk}
\usepackage{chngcntr}
\counterwithin{figure}{section}
\usepackage{xy}
\usepackage{amsthm}
\usepackage{xfrac}
\usepackage[dvipsnames]{xcolor}
\usepackage[
backend=bibtex,
giveninits=true,
style=alphabetic,
maxbibnames=99,mincitenames=1, maxcitenames=4]{biblatex}

\DeclareFieldFormat[article,periodical]{volume}{\mkbibbold{#1}}

\usepackage[urlcolor=pink, 
linkcolor=red, 
citecolor=Green,
colorlinks=true,
linktocpage=true]{hyperref}

\usepackage{tikz}
\usepackage{tikz-cd}
\tikzstyle{lien}=[->,>=stealth,rounded corners=5pt,thick]
\usepackage{geometry}
\usepackage{cleveref}

\crefformat{section}{\S#2#1#3} 
\crefformat{subsection}{\S#2#1#3}
\crefformat{subsubsection}{\S#2#1#3}

\newtheorem{thm}{Theorem}[section]
\newtheorem{prop}[thm]{Proposition}
\newtheorem{lem}[thm]{Lemma}

\theoremstyle{definition}
\newtheorem{defi}[thm]{Definition}
\theoremstyle{remark}
\newtheorem{rem}{Remark}[section]

\numberwithin{equation}{section}

\author{Julien Moy}
\title[$C^{1,\alpha}$ regularity of hypersurfaces of bounded NMC]{$ C^{1,\alpha}$ regularity of hypersurfaces of bounded nonlocal mean curvature in Riemannian manifolds}
\address{Département de Mathématiques et Applications, Ecole Normale Supérieure, Paris}
\email{julien.moy@ens.psl.eu}

\begin{document}

\begin{abstract}Let $(M,g)$ be a smooth connected Riemannian manifold. We show an improvement of flatness theorem for hypersurfaces of $M$ of bounded nonlocal mean curvature in the viscosity sense. It implies local $ C^{1,\alpha}$ regularity of these hypersurfaces provided that they are sufficiently flat. It extends a result of Caffarelli, Roquejoffre and Savin in the Euclidean setting to the case of arbitrary manifolds. \end{abstract}

\maketitle

\tableofcontents
		\section{Introduction}

	Nonlocal minimal surfaces were first introduced by Caffarelli, Roquejoffre and Savin in the seminal paper \cite{CRS}. These surfaces arise in the study of phase transition models with long range interactions. They are defined as boundaries of minimizers of a certain energy functional called the \textit{fractional perimeter}, depending on some parameter $s\in(0,1)$. The authors obtained a first regularity result: nonlocal minimal surfaces of $\mathbf R^n$ are $C^{1,\alpha}$, outside a closed singular set of Hausdorff dimension $n-2$.
	
	Caselli, Florit-Simon and Serra recently introduced in \cite{SSC} a notion of fractional perimeter for subsets of an arbitrary Riemannian manifold $(M,g)$. In their terminology, the boundary $\partial E$ of a subset $E\subset M$ is an \emph{$s$-minimal surface} if $E$ is a critical point of the fractional perimeter. They showed that closed manifolds of dimension $n\ge 2$ contain infinitely many of these surfaces (analogue in the fractional setting to the well known Yau's conjecture on classical minimal surfaces, recently proved in full generality by Song \cite{Song23}). It is argued in \cite{Cha+} that one should be able to construct classical minimal surfaces by taking limits of $s$-minimal surfaces, as the fractional parameter $s$ goes to $1$. Thus, nonlocal minimal surfaces may offer new ways to tackle questions about classical minimal surfaces.

	We focus here on a seemingly broader class of surfaces: those with bounded nonlocal mean curvature (NMC) in the viscosity sense (see Definition \ref{viscosity solution}). Our goal is to extend the improvement of flatness theorem for hypersurfaces of $\mathbf R^n$ with zero NMC proved in \cite{CRS}, to the case of hypersurfaces of bounded NMC in a Riemannian manifold. We stress out that our result holds in any smooth, connected, Riemannian manifold.
	\subsection{Fractional perimeter and nonlocal minimal surfaces in $\mathbf{R}^n$} We first describe the standard setup before diving into the setting of general manifolds. Let $s\in (0,1)$, that will be fixed in the paper. Given a measurable function $f:\mathbf R^n\to \mathbf R$, the (homogeneous) $H^{s/2}$-norm of $f$ can be written as 
	\[\|f\|_{H^{s/2}(\mathbf R^n)}^2=\int \!\!\int_{\mathbf R^n\times \mathbf R^n}\frac{|f(x)-f(y)|^2}{|x-y|^{n+s}}\ \mathrm dx\mathrm dy.\]
	We define the \emph{$s$-perimeter} of a measurable set $E\subset \mathbf{R}^n$ by
	\begin{equation}\label{eq:defpers}\mathrm{Per}_s(E):=\|\chi_E\|_{H^{s/2}(\mathbf R^n)}^2,\end{equation}
	where $\chi_E$ is the indicator function of the set $E$. This can be thought of as an energy taking into account long range interactions between points of $E$ and its complement $\mathcal CE$. Of course, without restriction on $E$, this quantity may be infinite. Assuming that $E$ is bounded with $C^2$ boundary is sufficient to guarantee convergence of \eqref{eq:defpers}. We point out that in the limit $s\uparrow 1^{-}$, the fractional perimeter $(1-s)\mathrm{Per}_s(E)$ converges, up to a multiplicative factor, to the usual perimeter $\mathrm{Per}(E)$, at least when $\partial E$ is smooth, see \cite{Dav}. Although half spaces have infinite $s$-perimeter, we still want them to minimize the $s$-perimeter in some sense. Thus, given an open subset $\Omega\subset \mathbf{R}^n$, we define the relative perimeter $\mathrm{Per}_s(E;\Omega)$ as the contribution of $\Omega$ to the $H^{s/2}$-norm of $\chi_E$, that is we only consider the contribution of pairs of points $(x,y)$ where at least one of them belongs to $\Omega$:
	\[\mathrm{Per}_s(E;\Omega):=\int \!\!\int_{\mathbf{R}^n\times \mathbf{R}^n\setminus \mathcal C\Omega\times \mathcal C\Omega}\frac{|\chi_E(x)-\chi_E(y)|}{|x-y|^{n+s}}\ \mathrm dx\mathrm dy.\]
	We can then define minimizers.
	\begin{defi}Given an open subset $\Omega\subset \mathbf R^n$, we say that a measurable subset $E\subset \mathbf R^n$ is a \emph{minimizer} of the $s$-perimeter in $\Omega$ if $\mathrm{Per}_s(E;\Omega)$ is minimal among sets which coincide with $E$ outside $\Omega$. In other words, $E$ is a minimizer in $\Omega$ if and only if for any subset $F\subset \mathbf{R}^n$ such that $F\cap \mathcal C\Omega=E\cap \mathcal C\Omega$ we have
		\[\mathrm{Per}_s(E;\Omega)\le \mathrm{Per}_s(F;\Omega).\]\end{defi}
	
	 Following the usual terminology, a point $x$ is said to belong to the \emph{interior} of $E$ if $|B_r(x)\backslash E|=0$ for some $r>0$. Up to modifying $E$ by a set of measure $0$, we shall always assume that $E$ contains its interior. Also, $\partial E$ is defined as the set of points $x$ such that $|E\cap B_r(x)|>0$ and $|\mathcal C E \cap B_r(x)|>0$ for any $r>0$. We point out that the interior of $E$ is open, while the boundary $\partial E$ is closed.
	
	The result of \cite{CRS} relies on an \textit{improvement of flatness} theorem for subsets of $\mathbf R^n$ of zero nonlocal mean curvature for the Euclidean metric in the viscosity sense (see \cite{CRS} for the definition in the Euclidean setup). This applies in particular to minimizers, by \cite[Theorem 5.1]{CRS}.

	\begin{thm}[Caffarelli--Roquejoffre--Savin {\cite{CRS}}]\label{thmCRS} Let $0<\alpha<s$. There exists $\sigma>0$, that may depend on $n,s$ and $\alpha$, such that the following holds. Assume $E\subset \mathbf{R}^n$ has zero nonlocal mean curvature in the viscosity sense in $B_1(0)$. Moreover, assume that $0\in \partial E$ and that $\partial E\cap B_1(0)$ is trapped in a $\sigma$-flat cylinder in the direction $x^n$ in $B_1(0)$, meaning
		\begin{equation}\label{eq: assumption mainthm 1}\{x^n \le -\sigma\}\cap B_1(0)\subset E \cap B_1(0)\subset \{x^n\le \sigma\}.\end{equation}
		Then $\partial E\cap B_{1/2}(0)$ is a $ C^{1,\alpha}$ graph in the $x^n$ direction.
	\end{thm}
	
	\begin{rem}\label{rem: uniform density } Minimizers enjoy some \emph{uniform density estimates}: there is a universal constant $c_\star$ such that if $0\in \partial E$ and $E$ is a minimizer of the $s$-perimeter in $B_r(0)$, then $|E\cap B_r(0)|\ge c_\star r^n$ and $|\mathcal CE\cap B_r(0)|\ge c_\star r^n$. Therefore, taking $\sigma$ small enough, one can replace the assumption \eqref{eq: assumption mainthm 1} by the \emph{a priori} weaker condition $0\in \partial E$ and $\partial E\cap B_1\subset \{|x^n|\le \sigma\}$ (this is how the Theorem is stated in \cite{CRS}). It is not clear however that surfaces of bounded or even zero NMC satisfy such estimates.\end{rem}
	In \cite{CaffaValdi}, if~$E$ is a minimizer, the authors managed to release the dependence of~$\sigma$ on the parameter~$s$ as~$s\uparrow 1$, hence obtaining a uniform improvement of flatness theorem. For $s$ close to~$1$, they showed that the singular set has Hausdorff dimension $n-8$. This result was improved in \cite{Fig}, where the authors showed smoothness of minimizers for~$s$ close to $1$, outside the singular set.
	
	\subsection{Nonlocal minimal surfaces in a Riemannian manifold} Let $(M,g)$ be a smooth, connected, orientable, $n$-dimensional Riemannian manifold. Based on \eqref{eq:defpers}, we aim to define the fractional perimeter of a subset $E\subset M$ as the $\frac s2$-Sobolev norm (this has yet to be defined) of the indicator function of $E$.
	
	\subsubsection{Fractional Sobolev spaces on Riemannian manifolds} In order to motivate the definition of the $H^{s/2}(M)$-norm, we first introduce the fractional Laplacian $\Delta^{s/2}$ on the manifold $(M,g)$. We point out that the following construction is valid for $s\in (0,2)$, although in the setting of the paper we will stick to $s\in (0,1)$. 
 
 Let $C^\infty_{\rm comp}(M)$ denote the space of smooth, compactly supported functions on $M$, and denote by $H_0^1(M)$ the completion of $C^\infty_{\rm comp}(M)$ for the norm
 \[\|f\|_{H_0^1(M)}:=\int_M |f|^2+|\nabla_g f|^2\mathrm dV.\]
Here $\nabla_gf $ denotes the gradient of $f$, and $\mathrm dV$ (that will sometimes be denoted $\mathrm dV_g$) is the Riemannian volume form on $M$. We define a dense subset of $L^2(M)$ by setting
\[\mathcal D=\{f\in H_0^1(M),~ \Delta f\in L^2\}.\]
Then, the Laplacian $\Delta_{|\mathcal D}$ is a nonnegative self-adjoint operator, called the \emph{Dirichlet Laplacian}. Note that when $(M,g)$ is geodesically complete, the operator $\Delta_{|C^\infty_{\rm comp}(M)}$ is essentially self-adjoint, \emph{i.e.} admits a unique self-adjoint extension. When $(M,g)$ is noncomplete, such an extension is not unique and one has to consider an initially larger domain. In the following, we just write $\Delta$ for the Dirichlet Laplacian, or $\Delta_g$ to stress dependence on the metric.
	
To define the fractional Laplacian, we start from the representation formula, valid for $\lambda>0$,
	\begin{equation}\label{lambdasigma}\lambda^{s/2}=\frac{1}{|\Gamma(-\frac s2)|}\int_0^{+\infty}(1-\mathrm{e}^{-\lambda t})\frac{\mathrm dt}{t^{1+s/2}}.\end{equation}
	Then, we formally define $\Delta^{s/2}$ by
	\begin{equation}\label{Bochnerdef} \Delta^{s/2}:=-(-\Delta)^{s/2}=\frac{1}{|\Gamma(-\frac s2)|}\int_0^{+\infty}(\mathrm{e}^{t\Delta}-\mathrm{Id})\frac{\mathrm dt}{t^{1+s/2}}.\end{equation}
	Here $\mathrm{e}^{t\Delta}$ denotes the \emph{heat semigroup}, see \cite[\S 4.3]{Grigorbook}. It is initially defined via functional calculus as an operator on $L^2(M)$, but it can also operate on the space of $C^\infty$-bounded functions, see \cite[Chapter 7]{Grigorbook}.
	
	We denote by $H(t,p,q)$ the \emph{heat kernel} on $M$, that is as usual the smallest fundamental solution to the heat equation on $M$. We may write $H_g$ instead of $H$ to stress the dependence on the metric when needed. Then, if $f:M\to \mathbf R$ is a smooth bounded function, we have
	\begin{equation}\label{semigroupe}\mathrm{e}^{t\Delta} f(p)=\int_{M} H(t,p,q)f(q)\mathrm dV(q).\end{equation}

	We recall that the function $H(t,p,q)$ is smooth and positive on $(0,+\infty)\times M\times M$, symmetric in space (meaning $H(t,p,q)=H(t,q,p)$) and satisfies
	\begin{itemize}\item $\displaystyle\frac{\partial H(t,p,q)}{\partial t} =\Delta H(t,p,q), \ \ \ \text{for all positive time $t$.}$
		\item $\displaystyle\int_M H(t,p,q)f(q)\mathrm dV(q)\underset{t\to 0}\longrightarrow f(p)$, \ \ \ \text{for any bounded function $f\in  C^{\infty}(M)$.}
	\end{itemize}
	The second condition is often rephrased by saying that $H(t,p,\cdot)\mathrm dV\to \delta_p$ as $t$ goes to $0$. For a general introduction to heat kernels on manifolds, we refer to \cite[Chapters 4,7--9]{Grigorbook}. 
	
	\begin{defi}[Stochastic completeness] We say that $(M,g)$ is \emph{stochastically complete} if the heat kernel $H(t,p,q)$ satisfies
		\[\int_M H(t,p,q)\mathrm dV(q)=1\]
		for any time $t>0$ and $p\in M$. \end{defi}
	\begin{rem}If $(M,g)$ is not stochastically complete, then the fractional Laplacian may have a wild behavior. In particular, the fractional Laplacian of a constant does not vanish. Nevertheless, stochastic completeness holds for a large class of manifolds. Indeed, by a result of Grigor'yan \cite{Grigoryan}, if for some $p\in M$ one has $\mathrm{Vol}(B_r^g(p))\le C \mathrm{e}^{ cr^2}$ then $(M,g)$ is stochastically complete. \end{rem}
	
	Let us go back to our representation formula. When $(M,g)$ is stochastically complete and $f:M\to \mathbf R$ is a smooth bounded function we can write
	\[(\mathrm{e}^{t\Delta}-I)f(p)=\int_{M} H(t,p,q)\big(f(q)-f(p)\big)\mathrm dV(q).\]
	Integrating against $\frac{\mathrm dt}{t^{1+s/2}}$ yields
	\begin{equation}\label{eq: fracLap pv}-(-\Delta)^{s/2} f(p)=\mathrm{p.v.}\int_M K(p,q)\big(f(q)-f(p)\big) \mathrm dV(q),\end{equation}
	where $K(p,q)$, that we refer to as the kernel of the fractional Laplacian, is defined by
	\begin{equation}\label{defKpq} K(p,q):=\int_{\mathbf R_+} H(t,p,q)\frac{\mathrm dt}{t^{1+s/2}}.\end{equation}
	The principal value in \eqref{eq: fracLap pv} has to be understood as the limit of integrals over $M\setminus B^g_\varepsilon(p)$ (here $B^g_\varepsilon(p)$ is the metric ball of radius $\varepsilon$ centered at $p$), when $\varepsilon$ goes to $0$. Again, we may write $K_g(p,q)$ instead of $K(p,q)$ to stress the dependence on the metric. Note that in the case $(M,g)=(\mathbf{R}^n,\mathrm{eucl})$ we recover $K(x,y)=\alpha_{n,s} |x-y|^{-(n+s)}$, where $\alpha_{n,s}$ is an explicit positive constant.
	
	Note that the right-hand side of \eqref{eq: fracLap pv} makes sense on any manifold, even when $M$ is not stochastically complete, though in this case it does not coincide with \eqref{Bochnerdef}. Hence, stochastic completeness is a rather natural hypothesis in our setting, even though our results do not require it. 
	
	We refer to \cite{caselli2024fractional,BGS} for other equivalent definitions of the fractional Laplacian on manifolds, either based on a spectral approach on closed manifolds, or on a Caffarelli--Silvestre extension problem. These however aren't used in our paper.

	\begin{rem}[Change of variable and rescaling] \label{rem:scaling} It will be useful to keep in mind that if $(M,g)$ is a Riemannian manifold and $\varphi:N\to M$ is a smooth diffeomorphism, then $H_{\varphi^* g}(t,\varphi^{-1}(p),\varphi^{-1}(q))=H_g(t,p,q)$, and if $\lambda>0$ then $H_{\lambda^2 g}(t,p,q)=\lambda^{-n}H_g(t/\lambda^2,p,q)$. Consequently, we also have $K_{\varphi^* g}(\varphi^{-1}(p),\varphi^{-1}(q))=K_g(p,q)$ and $K_{\lambda^2 g}(p,q)=\lambda^{-(n+s)} K_g(p,q)$.  \end{rem}

	\begin{defi}
		We define the $H^{s/2}$-norm of a measurable function $f:M\to \mathbf R$ by
		\begin{equation}\label{eqdefSobolevnorm}\|f\|_{H^{s/2}(M)}^2:=\int\!\!\int_{M\times M}\big(f(p)-f(q)\big)^2K(p,q) \mathrm dV(p)\mathrm dV(q).\end{equation}\end{defi}
	Notice that, thanks to the symmetry property $K(p,q)=K(q,p)$, if $(M,g)$ is stochastically complete then $\|f\|_{H^{s/2}(M)}^2=\langle f,(-\Delta)^{s/2}f\rangle_{L^{2}(M)}$ for any smooth compactly supported function $f$. 
	\subsubsection{Fractional perimeter, $s$-minimal sets and nonlocal mean curvature.}
	
	Similarly to the Euclidean setting, with \eqref{eqdefSobolevnorm} at hand, we define the $s$-perimeter of a measurable subset $E\subset M$ by
	\[\mathrm{Per}_s(E):= \|\chi_E\|_{H^{s/2}(M)}^2=\int\!\!\int_{M\times M} K(p,q)|\chi_E(p)-\chi_E(q)|\mathrm dV(p)\mathrm dV(q).\]
	
	Now we define $s$-minimal sets, and nonlocal mean curvature of subsets of arbitrary manifolds.
	
	\begin{defi}[Nonlocal minimal hypersurfaces] Assume $(M,g)$ is a closed manifold. We say that a measurable subset $E\subset M$ is \emph{$s$-minimal} if its fractional perimeter is finite and if it is a critical point of the fractional perimeter functional, that is, for any $C^\infty$ vector field $X$ on $M$ we have
		\[\left.\frac{\mathrm d}{\mathrm dt}\right|_{t=0}\mathrm{Per}_s(\phi_X^t(E))=0,\]
		where $\phi_X^t$ denotes the flow of $X$. The boundary $\partial E$ is called a \emph{$s$-minimal hypersurface}. 
	\end{defi}
	
	\begin{rem} As pointed out in \cite[Remark 1.8]{SSC}, if $\mathrm{Per}_s(E)<\infty$ then the map $t\mapsto \mathrm{Per}_s(\phi_X^t(E))$ is smooth, thus the definition above makes sense. There is a local version of this definition, that proves useful when working in noncompact manifolds, by considering perturbations of $E$ in an open subset $\Omega$, see \cite[\S 1.3]{SSC} for details.\end{rem}
	
	\begin{defi}[Nonlocal mean curvature] Consider a measurable subset $E\subset M$. If $p$ belongs to $\partial E$ we define (when it makes sense) the \emph{nonlocal mean curvature} (NMC) of $E$ at $p$ by the formula
		\begin{equation}\label{eq: def NMC} \mathrm H_s[E](p):=\mathrm{p.v.}\int_M\big(\chi_E(q)-\chi_{\mathcal CE}(q)\big) K(p,q)\mathrm dV(q),\end{equation}
		where the principal value is understood as the limit as $\varepsilon\to 0$ of integrals over $M\backslash B_\varepsilon^g(p)$, where $B_\varepsilon^g(p)$ denotes the metric ball of radius $\varepsilon$ centered at $p$. We may write $\mathrm{H}_s^g[E]$ for the NMC to stress the dependence on the metric. 
		
	\end{defi}
	
	\begin{rem}[Rescaling and invariance by isometry]\label{rem: scaling H_s} By Remark \ref{rem:scaling} we immediately get that if $E$ is a measurable subset of $(M,g)$ then for any $\lambda>0$ we have $\mathrm H_s^{\lambda^2 g}[E](p)=\lambda^{-s}\mathrm H_s^g[E](p)$; and if $\varphi:N\to M$ is a smooth diffeomorphism, then $\mathrm H_s^{\varphi^*g}[\varphi^{-1}(E)](\varphi^{-1}(p))=\mathrm H_s^g[E](p)$.\end{rem}

	Similarly to the Euclidean case, we say that $E$ is a \emph{minimizer} in $\Omega$ if $\mathrm{Per}_s(E;\Omega)\le \mathrm{Per}_s(F;\Omega)$ for any subset $F\subset M$ such that $F\cap \mathcal C\Omega=E\cap \mathcal C\Omega$. On closed manifolds, one can show that minimizers are $s$-minimal sets, but the converse is false, even for stable $s$-minimal sets \cite[\S 1.1]{Cha+}. It can be shown for smooth sets that nonlocal mean curvature arises as the first variation of the fractional perimeter functional. In particular, $s$-minimal smooth sets have zero nonlocal mean curvature. However, when no information on the regularity of $\partial E$ is known, $\mathrm H_s[E]$ is \emph{a priori} not well-defined, and we need to introduce nonlocal mean curvature in a weak sense. 
	
	\subsection{NMC boundedness in the viscosity sense}

	\begin{defi}Following the terminology of \cite{SSC}, we say that a measurable subset $E\subset M$ has an interior tangent ball at $p\in \partial E$ if there exists a smooth diffeomorphism $\psi$ from $B_1(0)\subset \mathbf{R}^n$ onto a neighborhood $V$ of $p$ in $M$, such that $\psi(0)=p$ and $V^+:=\psi(B_1^+)\subset E$, where $B_1^+=\{(x',x^n)\in B_1(0), \ x^n>0\}$. 
 
 Similarly, we say that $E$ has an exterior tangent ball at $p\in \partial E$ if there exists such a $\psi$ with instead $V^-:=\psi(B_1^-)\subset \mathcal CE$, where $B_1^{-}=\{(x',x^n)\in B_1(0),  \ x^n<0\}$. \end{defi}
	
	A notable feature is that a measurable subset $E\subset M$ has well-defined NMC at $p\in \partial E$ whenever it has an interior or exterior tangent ball at $p$. In the first case, one can show that~$\mathrm H_s[E](p)\in (-\infty,+\infty]$, while in the second case ${\mathrm H_s[E](p)\in [-\infty,+\infty)}$. This fact is not entirely trivial and is the content of Proposition \ref{prop: NMC bien définie} for manifolds that are quasi-isometric to $\mathbf R^n$ (with control on the first derivatives of the metric coefficients), and Proposition \ref{prop: NMC bien définie bis} in the general case. A proof for $(M,g)=(\mathbf R^n,\mathrm{eucl})$ can be found in the work of Cabré \cite{cabré}, who also simplified the proof that minimizers of the fractional perimeter have zero NMC in the viscosity sense (for the Euclidean metric).

	Now we can define sets with bounded nonlocal mean curvature in the viscosity sense.

	\begin{defi}\label{viscosity solution}Let $E$ be a measurable subset of $(M,g)$, and let $\Omega\subset M$. We say that $E$ has NMC bounded by $C_0$ \emph{in the viscosity sense} in $\Omega$ if and only if the two following conditions are satisfied:
		
		\begin{itemize}\item whenever $E$ as an interior tangent ball at $p\in \partial E\cap \Omega$, we have $\mathrm H_s[E](p)\le C_0$.
			\item whenever $E$ has an exterior tangent ball at $p\in \partial E\cap \Omega$, we have $\mathrm H_s[E](p)\ge -C_0$.
		\end{itemize}
	\end{defi}
	
	There is an equivalent definition of NMC boundedness used in \cite{SSC}, that may appear weaker at first but yields the same notion. We postpone the proof of the equivalence to \S \ref{subsec: NMC bis}.
	
	\begin{prop}\label{prop: equiv def NMC}A measurable subset $E\subset M$ has NMC bounded by $C_0$ in the viscosity sense in $\Omega$ if and only if for any smooth diffeomorphism $\psi:B_1(0)\subset \mathbf R^n\to V\subset M$ such that $\psi(0)=p\in \partial E\cap \Omega$, letting $F=V^+\cup (E\backslash V)$, the two following conditions are verified:
		
		\begin{itemize} \item if $V^+:=\psi(B_1^+)\subset E$, then  $\mathrm{H}_s[F](p)\le C_0$.
			\item if $V^-:=\psi(B_1^-)\subset \mathcal CE$, then $\mathrm{H}_s[F](p)\ge -C_0$. \end{itemize}\end{prop}

	We expect that minimizers, in our setup, have zero nonlocal mean curvature in the viscosity sense (\emph{i.e.} have NMC bounded by $0$ in the viscosity sense in our terminology). The result should follow by adapting the calibration argument of Cabr\'e \cite{cabré}.
	
	 Note that since the nonlocal mean curvature arises as the first variation of the fractional perimeter, if $E$ has smooth boundary, then $E$ is $s$-minimal if and only if $E$ has zero NMC. However, without \emph{a priori} regularity assumptions on $\partial E$, the equivalence above seems more mysterious (replace $E$ has zero NMC by $E$ has zero NMC in the viscosity sense), as the method of \cite{cabré} doesn't seem to adapt easily to sets that are merely critical points of the fractional perimeter. Yet, let us mention the following result: the authors of \cite{SSC} showed that any closed manifold $(M,g)$ contains infinitely many $s$-minimal hypersurfaces, constructed by taking limits as $\varepsilon\to 0^+$ of solutions with bounded Morse index to the fractional Allen--Cahn equation. The hypersurfaces that they obtained have zero NMC in the viscosity sense. However, it is not known whether any $s$-minimal hypersurface can be obtained in such a way, hence the question remains open.

	
	\subsection{Main results} 
	
	We are now ready to state our main theorem.

	\begin{thm} \label{mainthm} Let $0<\alpha<s<1$ and $C_0>0$. There exists positive constants $\sigma$ and $C_1$ that may depend on $n,s,\alpha$ and $C_0$, such that the following holds. Let $r>0$. Let $g=(g_{ij})$ be a smooth Riemannian metric on $\mathbf{R}^n$, and $E$ be a measurable subset of $\mathbf R^n$, such that
			
			\begin{itemize}
				\item $\frac 12 \le g\le 2$ (in the sense of quadratic forms) and $r\|\nabla g_{ij}\|_{L^\infty(\mathbf R^n)}\le 1$ for any~${i,j\in \{1,\ldots,n\}}$.
				\item The point $0$ belongs to $ \partial E$, and the boundary $\partial E$ is trapped in a $\sigma$-flat cylinder in the direction $x^n$ in $B_r(0)$, meaning
				\begin{equation}\label{hypothèse trapping}\{x^n\le -\sigma r\}\cap B_r(0)\subset E\cap B_r(0)\subset \{x^n\le \sigma r\}.\end{equation}
				\item The set $E$ has nonlocal mean curvature bounded by $C_0r^{-s}$ in the viscosity sense (for the metric $g$) in the Euclidean ball $B_r(0)$.
			\end{itemize}
			Then $\partial E$ is a $ C^{1,\alpha}$ graph in the direction $x^n$ in $B_{r/2}'(0)\times [-\sigma r,\sigma r]$ with uniform estimates, meaning that
			\[\partial E\cap \big(B_{r/2}'(0)\times [-\sigma r,\sigma r]\big)=\big \{(x',f(x')), \  x'\in B_{r/2}'(0)\big \},\]
			for some $C^{1,\alpha}$ function $f:B_{r/2}'(0)\to [-\sigma r,\sigma r]$ satisfying
			\[\|\nabla f\|_{C^\alpha(B_{r/2}'(0))}\le C_1r^{-\alpha}.\]\end{thm}
	
	\begin{rem}The fact that we ask bounds of NMC by $C_0r^{-s}$ is no surprise since the nonlocal mean curvature scales like $r^{-s}$ when we rescale the metric by a factor $r$, as follows from Remark \ref{rem: scaling H_s}. \end{rem}

	As we will be dealing with subsets $E$ of general Riemannian manifolds, we need an analogous assumption to the trapping hypothesis \eqref{hypothèse trapping}. It is not clear however to see what it means for a subset $E\subset M$ to be trapped in a $\sigma$-flat cylinder in the metric ball $B_r^g(p)$, when $r$ is arbitrary large. However, when $r$ is small enough, we can always map $B_r^g(p)$ quasi-isometrically to the Euclidean ball $B_r(0)$, \emph{e.g.} by taking Riemannian normal coordinates, and consider cylinders in these coordinates.

	From the previous discussion, let us introduce the following definition, from \cite{SSC}.
	\begin{defi}[Local flatness assumption]\label{Local flatness} A manifold $(M,g)$ satisfies a \emph{flatness assumption} of order $1$ at scale $r$ around $p$, with parametrization $\varphi$, denoted $\mathrm{FA}_1(M,g,p,r,\varphi)$, if there exists a smooth diffeomorphism $\varphi:B_r(0)\subset \mathbf R^n\to V\subset M$ with $\varphi(0)=p$, such that if $\varphi^*g=(g_{ij})$ denotes the pullback metric, then $\frac 12\le \varphi^*g\le 2$ in $B_r(0)$, and for any $i,j\in\{1,\ldots,n\}$ we have
		\[r \|\nabla g_{ij}\|_{L^{\infty}(B_r(0))}\le 1.\]\end{defi}
	\begin{rem}We point out that for any $p\in M$, it is always possible to find a small radius $r$ and a diffeomorphism $\varphi$ such that $\mathrm{FA}_1(M,g,p,r,\varphi)$ holds, \emph{e.g.} by taking for $\varphi$ the exponential map $\exp_p:B_r(0)\subset \mathbf R^n\to B_r^g(p)\subset M$ for $r$ small enough (here we fix an identification between $\mathbf R^n$ and the tangent space $T_pM$). \end{rem}
	
	\begin{rem}There are two notions of flatness at play here. Firstly, we deal with flatness of cylinders in $\mathbf R^n$, defined as the ratio between the height and the diameter of the base. Secondly, roughly speaking, a manifold $M$ satisfies a flatness assumption of order $1$ at scale $r$ around $p$ if $B_r^g(p)$ is quasi-isometric to the Euclidean ball $B_r(0)$ (with additional control on first derivatives of the metric coefficients). Since these notions apply to distinct objects, we hope that it causes no confusion to the reader. \end{rem}
	
	Now we can formulate our second result, that extends Theorem \ref{mainthm} to the setting of arbitrary manifolds.
	
	\begin{thm}\label{maincor} Let $0<\alpha<s<1$ and $C_0>0$. There exists positive constants $\sigma$ and $C_1$ depending on $C_0,n,s$ and $\alpha$ such that the following holds. Assume $(M^n,g)$ is a smooth $n$-dimensional Riemannian manifold satisfying a flatness assumption $\mathrm{FA}_1(M,g,p,r,\varphi)$ for some $p\in M$, and $E\subset M$ satisfies:
		\begin{itemize} 
			\item  $E$ has nonlocal mean curvature (for the metric $g$) bounded by $C_0r^{-s}$ in the viscosity sense in $\varphi(B_r(0))$.
			\item  The point $p$ belongs to $ \partial E$ and $\varphi^{-1}(\partial E)$ is trapped in a $\sigma$-flat cylinder in $B_r(0)$, i.e. there is some $\nu\in \mathbf S^{n-1}$ such that
			\[\{x\cdot \nu \le -\sigma r\}\cap B_r(0)\subset \varphi^{-1}(E)\cap B_r(0)\subset \{x\cdot \nu \le \sigma r\}.\] 
		\end{itemize}
		Then, $\varphi^{-1} (\partial E)$ is the graph of a $C^{1,\alpha}$ function $f$ in the direction $\nu$ in $B_{r/2}'(0)\times [-\sigma r,\sigma r]$, with \[\|\nabla f\|_{C^\alpha(B_{r/2}'(0))}\le C_1 r^{-\alpha}.\]
	\end{thm}
	
	
	\subsection{Notations}
	If $\bullet$ is a set of parameters, we use the notations $C_\bullet,c_\bullet$ to denote positive constants depending only on $\bullet$, that may change from line to line. We may write for example $C_{n,s}$ for a constant depending only on $n,s$.
	
	To avoid heavy notations, if $A,B$ and $\Omega$ are three sets, we say that $A\subset B$ in $\Omega$ if and only if $A\cap \Omega\subset B\cap \Omega$.
	
	We will often denote points of $\mathbf{R}^n$ by $x=(x',x^n)\in \mathbf{R}^{n-1}\times \mathbf{R}$. We denote by $B_r'(x)\subset \mathbf R^{n-1}$ the ball $\{|y'-x'|<r\}$.
	
	If $g$ is a metric on $\mathbf R^n$, we write shortly $|g|:=|\det g|$, so that the volume form $\mathrm dV_g(z)$ is given by $\sqrt{|g(z)|}\mathrm dz$.
	
	If $(M,g)$ is a Riemannian manifold, the metric ball of radius $r$ centered at $p$ is denoted by $B_r^g(p)$. We drop the superscript when $g$ is the Euclidean metric on $\mathbf R^n$.
	
	If $M$ is a manifold and $\varphi$ is a  smooth diffeomorphism from $B_R(0)\subset \mathbf R^n$ to a subset $V\subset M$, we denote $V_r(q):=\varphi(B_r(\varphi^{-1}(q)))$ when it makes sense. In particular, if $\varphi(0)=p$ then $V_r(p)=\varphi(B_r(0))$ and $V_R(p)=V$.

	\subsection{Organization of the paper}

		In Section \ref{Euler Lagrange equation} we start by recalling pointwise Gaussian upper bounds on heat kernels associated to metrics on $\mathbf R^n$ that are comparable to the Euclidean metric. These are standard and follow for example from the fact that such manifolds satisfy a \emph{relative Faber--Krahn inequality}. Then, we provide a crucial integral estimate on the difference $K_g(y,\cdot)\sqrt{|g|}-K_{g(y)}(y,\cdot)\sqrt{|g(y)|}$ when $g$ satisfies the assumptions of Theorem \ref{mainthm}.
		\vspace{3pt}
		
		In Section \ref{sec:freeze}, we establish that if $E\subset \mathbf R^n$ has finite NMC at some point $y\in \partial E$ for the metric $g$, then it has finite NMC at the point $y$ for the constant metric $g(y)$. The proof relies on the estimates of the previous section. It allows rephrasing the NMC boundedness assumption only in terms of constant metrics. Working with constant metrics offers the benefit to deal with kernels that are invariant under the symmetry $x\mapsto 2y-x$. This property allows in turn to exploit the hypothesis that $\partial E$ is trapped in flat cylinders to obtain cancellations when estimating integrals in the proof of the improvement of flatness Theorem (Theorem \ref{improvement}).
		\vspace{3pt}
		
		Section \ref{Improvement of flatness} is devoted to the proof of the improvement of flatness theorem for hypersurfaces of $(\mathbf R^n,g)$ of bounded NMC in the viscosity sense, when $g$ satisfies the assumptions of Theorem \ref{mainthm}. The proof follows closely that of \cite{CRS}. However, some subtleties arise because we are dealing with nonconstant metrics and sets of bounded instead of zero NMC, hence the necessity to work out again the proof. A standard argument then allows deducing Theorem \ref{mainthm} from Theorem \ref{improvement}.
		\vspace{3pt}
		
		In Section \ref{fromto} we go back to the setting of an arbitrary Riemannian manifold $(M,g)$. We prove some integral estimates for the heat kernel $H(t,p,\cdot)$ when $M$ satisfies a flatness assumption at scale $r$ around $p$. These estimates allow translating the problem from $(M,g)$ to a compact perturbation of $(\mathbf R^n,\mathrm{eucl})$, thus reducing Theorem \ref{maincor} to Theorem \ref{mainthm}.

	\subsection*{Acknowledgment}
	
	I want to thank Joaquim Serra for introducing me to this topic and for helpful discussions and advises.

	\section{Heat kernel estimates in $\mathbf R^n$}\label{Euler Lagrange equation}
	In this section, we provide various heat kernel estimates for manifolds satisfying the assumptions of Theorem \ref{mainthm}. They will be used extensively in the next sections.
	
	\subsection{Kernel of the $s/2$-Laplacian for constant metrics on $\mathbf{R}^n$} We start by giving an explicit formula for $K_g(x,y)$, when $(M,g)$ is $\mathbf{R}^n$ equipped with a constant metric.
	
	\begin{lem} \label{kernelconstantmetric}Let $g$ be a $n\times n$ positive definite symmetric matrix, viewed as a constant metric on $\mathbf R^n$. The kernel of the $\frac s2$-Laplacian associated to this metric is 
		\[K_g(x,y)=\alpha_{n,s} |x-y|_g^{-(n+s)}, \]
		where $|\cdot|_g$ is the norm associated to $g$ and $\alpha_{n,s}$ is an explicit positive constant independent of $g$.\end{lem}

		\begin{proof}This comes down to finding an expression for the heat kernel $H_g(t,x,y)$. It can be computed directly by explicitly solving the heat equation, and one finds
			\begin{equation}\label{eq:heatconstant}H_g(t,x,y)=\frac{1}{(4\pi t)^{n/2}}\exp\Big(-\frac{|x-y|_g^2}{4t}\Big).\end{equation}
			Alternatively, assuming that one already knows the expression of the kernel associated to the Euclidean metric, it follows from Remark \ref{rem:scaling} by setting $(M,h)=(\mathbf R^n,\mathrm{eucl})$ and $\varphi=g^{1/2}:\mathbf R^n\to \mathbf R^n$.
	\end{proof}

	\subsection{Heat kernels bounds for metrics comparable to the Euclidean metric}   Now, we consider a smooth Riemannian metric $g$ on $\mathbf R^n$ that is comparable to the Euclidean metric in the sense that $\frac 12 \le g\le 2$ (we make no assumptions on the derivatives of $g$ yet). There is a convenient tool to obtain heat kernel upper bounds for such metrics, called the relative Faber--Krahn inequality.
		
		\begin{defi}[see {\cite[\S 15.6]{Grigorbook}}] A $n$-dimensional non-compact manifold $(M,g)$ satisfies the \emph{relative Faber--Krahn inequality (RFK)} if there exists a positive constant $b>0$  such that for any $x\in M$ and radius $r$, for any relatively compact open subset $\Omega\subset\subset B^g_r(x)$, we have
			\[\lambda_1(\Omega)\ge \frac{b}{r^2}\left(\frac{V(x,r)}{\mathrm{Vol}(\Omega)}\right)^{2/n},\] 
			where $\lambda_1(\Omega)$ denotes the first eigenvalue of the Dirichlet Laplacian $-\Delta_{\Omega}$. Here, $V(x,r)$ denotes the volume of the metric ball $B_r^g(x)$. \end{defi}
		We shall use the following characterization from \cite[Theorem 15.21]{Grigorbook} (see also \cite[Proposition 5.2]{grigoribero} for details on the dependencies of the constants involved):
		\begin{prop}\label{prop : Faber Krahn and gaussian bounds} Let $(M,g)$ be a complete non-compact manifold. Then $(M,g)$ satisfies the relative Faber--Krahn inequality (RFK) if and only if the two following properties hold:
			\begin{itemize} \item $(M,g)$ enjoys the volume doubling property for balls
				\begin{equation}\label{VD}\frac{V(x,R)}{V(x,r)}\le C_{1}\left(\frac{R}{r}\right)^{n}\ \ \ \ \text{for all $0<r\le R$}.\end{equation}
				\item There are Gaussian upper bounds on the heat kernel
				\begin{equation}\label{heat kernel upper bounds}H_g(t,x,y)\le \frac{C_2}{V(x,\sqrt t)}\mathrm{e}^{-\frac{d(x,y)^2}{Dt}}.\end{equation}
			\end{itemize}
			Moreover we have the following dependencies: $C_1=C_1(b,n)$, $D$ is any number $>4$ and $C_2=C_2(D,b,n)$.
		\end{prop}
		A notable property of (RFK) is its stability under quasi-isometries:
		
		\begin{lem}[{\cite[Corollary 15.15]{Grigorbook}}] \label{lem : RFK under quasi isometry} Let $(M,g)$ be a $n$-dimensional complete noncompact manifold, satisfying (RFK) for some $b>0$. Let $h$ be another metric on $M$, comparable to $g$ in the sense that $\frac 12 h\le g\le 2h$. Then, the manifold $(M,h)$ satisfies (RFK) with some different constant $b' =c_{n}b$. \end{lem}

 For our concerns, we are only interested in manifolds $(\mathbf R^n,g)$ with $\frac 12\le g\le 2$. In this case,  $V(x,\sqrt t)$ is comparable to $t^{n/2}$ and we have:
	
	\begin{prop}\label{FBmetric} Let $g$ be a smooth Riemannian metric on $\mathbf R^n$ satisfying $\frac 12\le g\le 2$. Then the heat kernel $H_g$ enjoys Gaussian upper bounds
			\[H_g(t,x,y)\le \frac{C}{t^{n/2}}\mathrm{e}^{-\frac{|x-y|_g^2}{D t}}.\]
			Here $D$ is any constant $>4$, and $C$ is a constant depending only on $D$ and $n$. Consequently, we have the following upper bound on the kernel of the fractional Laplacian, for some constant $C_{n,s}>0$:
			\[K_g(x,y)\le \frac{C_{n,s}}{|x-y|^{n+s}}.\]\end{prop}
		
		\begin{proof} From the expression of the standard heat kernel on $\mathbf R^n$, we see that $(\mathbf R^n,\mathrm{eucl})$ satisfies the items \eqref{VD} and \eqref{heat kernel upper bounds} of Proposition \ref{prop : Faber Krahn and gaussian bounds}. Consequently, $(\mathbf R^n,\mathrm{eucl})$ satisfies (RFK). Since $(\mathbf R^n,g)$ is quasi-isometric to $(\mathbf R^n,\mathrm{eucl})$, by Lemma \ref{lem : RFK under quasi isometry}, $(\mathbf R^n,g)$ satisfies (RFK), and we can appeal to Proposition \ref{prop : Faber Krahn and gaussian bounds} in the other direction to obtain the sought upper bound on the heat kernel associated to the metric $g$. The bound on $K_g(x,y)$ follows by a simple integration, observing that $|x-y|\le 2|x-y|_g$. \end{proof}
	
	\subsection{Freezing the metric} Now we consider a metric $g$ satisfying the assumptions of Theorem \ref{mainthm}. Fix $y\in \mathbf R^n$. The kernel $x\mapsto K_g(x,y)\sqrt{|g(x)|}$ associated to the metric $g$ is singular at $x=y$. However, we will show that it has the same singularity at $x=y$ as the kernel $x\mapsto K_{g(y)}(x,y)\sqrt{|g(y)|}$ associated to the constant metric $g(y)$, in the sense that the difference between the two kernels is integrable with respect to $x$. The purpose of this section is to show
	\begin{prop}\label{prop:changer noyau} Consider a smooth metric $g$ on $\mathbf R^n$, satisfying $\frac 12\le g\le 2$ and $r\|\nabla g_{ij}\|_{L^\infty(\mathbf R^n)}\le 1$ for any $i,j\in \{1,\ldots,n\}$. Then there is a constant $C_{n,s}$ depending only on $n,s$ such that for any $y\in \mathbf R^n$ we have
		\begin{equation}\label{diff noyaux}\int_{\mathbf R^n} \big|K_g(x,y)\sqrt{|g(x)|}-K_{g(y)}(x,y)\sqrt{|g(y)|}\big|\mathrm dx\le C_{n,s}r^{-s}.\end{equation}
	\end{prop}
	It allows us to rephrase the NMC boundedness assumption in a more practical way in the next section, by involving kernels associated to constant metrics. Proposition \ref{prop:changer noyau} will essentially follow from the next lemma, that provides an integral estimate on the difference between the heat kernels associated to the metrics $g$ and $g(y)$.
	
	\begin{lem} \label{effacement singularites}Assume $\frac 12\le g\le 2$ and $r\|\nabla g_{ij}\|_{L^\infty(\mathbf R^n)}\le 1$. Then, there is a constant $C_n>0$ depending only on $n$ such that for any $t>0$ and $y\in \mathbf R^n$,
		\[\int_{\mathbf R^n} \big|H_g(t,y,x)-H_{g(y)}(t,y,x)\big|\mathrm dx\le C_n\min \Big(1,\frac{\sqrt t}{r}\Big).\]
		Consequently, there is a constant $C_{n,s}$ depending only on $n$ and $s$ such that for any $y\in \mathbf R^n$,
		\[\int_{\mathbf R^n} \big|K_g(x,y)-K_{g(y)}(x,y)\big|\mathrm dx\le C_{n,s}r^{-s}.\]\end{lem}
	Before diving in the proof of Lemma \ref{effacement singularites}, we provide a representation formula for the difference $H_g(t,y,x)-H_{g(y)}(t,y,x)$.
	
	\begin{lem}\label{lem: Duhamel}Set $f(t,x):= H_g(t,y,x)-H_{g(y)}(t,y,x)$ and $F(t,x):=(\Delta_g-\Delta_{g(y)}) H_{g(y)}(t,x,y)$. Then we have the representation formula
		\begin{equation}\label{eq: representation f} f(t,x)=\int_0^t \int_{\mathbf R^n} H_g(t-\tau,x,z)F(\tau,z)\mathrm dV_g(z).\end{equation}\end{lem}
	The result essentially follows from Duhamel's formula, but one has to be careful because $f$ is not continuous at $t=0$. Still, $f(t,\cdot)$ converges to $0$ in the sense of distributions as $t\to 0^+$.
	\begin{proof}Start by noticing that
		\[\partial_t f(t,x)=\Delta_g f(t,x)+ F(t,x).\]
		Fix $t>0$, and take $\varepsilon\in (0,t)$. Since $f$ is smooth on $(0,+\infty)\times \mathbf R^n$, we have by Duhamel's formula:
		\[f(t,x)=\int_{\mathbf R^n} H_g(t-\varepsilon,x,z) f(\varepsilon,z)\mathrm dV_g(z)+\int_\varepsilon^{t}\int_{\mathbf R^n} H_g(t-\tau,x,z) F(\tau,z)\mathrm dV_g(z)\mathrm d\tau.\]
		Note that we cannot set $\varepsilon=0$ because $f$ is not continuous at $t=0$. Still, to obtain \eqref{eq: representation f} we just have to show
		\[\int_{\mathbf R^n} H_g(t-\varepsilon,x,z) f(\varepsilon,z)\mathrm dV_g(z)\underset{\varepsilon\to 0}\longrightarrow 0.\]
		First we need to get rid of the dependence in $\varepsilon$ in $H_g(t-\varepsilon,x,z)$. To achieve this, write
		\begin{equation}\label{toachieve} H_g(t-\varepsilon,x,z) f(\varepsilon,z)= H_g(t,x,z)f(\varepsilon,z)+\big(H_g(t-\varepsilon,x,z)-H_g(t,x,z)\big) f(\varepsilon,z).\end{equation}
		To handle the first term of the right-hand side of \eqref{toachieve}, notice that
		\begin{equation}\label{deux termes} \begin{array}{ll}\displaystyle \int_{\mathbf R^n}H_g(t,x,z)f(\varepsilon,z)\mathrm dV_g(z)& =\displaystyle\int_{\mathbf R^n} H_g(t,x,z) H_g(\varepsilon,y,z)\mathrm dV_g(z)\\ & -\displaystyle\int_{\mathbf R^n} \Big(\frac{\sqrt{|g(z)|}}{\sqrt{|g(y)|}}H_g(t,x,z)\Big) H_{g(y)}(\varepsilon,y,z) \mathrm dV_{g(y)}(z).\end{array}\end{equation}
		Recall that $H_g(\varepsilon,y,\cdot)\mathrm dV_g$ and $H_{g(y)}(\varepsilon,y,\cdot)\mathrm dV_{g(y)}$ both converge to $\delta_y$ as $\varepsilon\to 0$. Therefore, both integrals in the right-hand side of \eqref{deux termes} converge to $H_g(t,x,y)$ as $\varepsilon\to 0$, hence their difference goes to $0$. To handle the second term in the right-hand side of \eqref{toachieve}, we use the fact that the heat kernels have mass $1$ (against the appropriate volume forms, but here $2^{-n/2}\le \sqrt{|g|} \le 2^{n/2}$) to bound
		\[\left|\int_{\mathbf R^n} \big(H_g(t-\varepsilon,x,z)-H_g(t,x,z)\big) f(\varepsilon,z)\mathrm dz\right|\le C_n \|H_g(t-\varepsilon,x,\cdot)-H_g(t,x,\cdot)\|_{L^\infty(\mathbf R^n)}.\]
		This quantity goes to $0$ as $\varepsilon\to0$, as follows from the continuity of $H_g(\cdot,x,\cdot)$ on $(0,+\infty)\times \mathbf R^n$ and the decay properties of the heat kernel. This concludes the proof.
	\end{proof}
	Let us now turn to the proof of Lemma \ref{effacement singularites}.
	
	\begin{proof}[Proof of Lemma \ref{effacement singularites}]
		With notations of Lemma \ref{lem: Duhamel}, we want to show
		\begin{equation}\label{withnotationgoal}\int_{\mathbf R^n} |f(t,x)|\mathrm dx\le C_n\min\Big(1,\frac{\sqrt t}{r}\Big).\end{equation}
		We start from the representation formula \eqref{eq: representation f}. By the triangular inequality, we have
		\[|f(t,x)|\le \int_0^t \int_{\mathbf R^n} H_g(t-\tau,x,z)|F(\tau,z)|\mathrm dV_g(z).\]
		We integrate this inequality over $x\in\mathbf R^n$, recalling that $H_g(\tau,\cdot,z)\sqrt{|g|}$ has mass $1$ and $\sqrt{|g|}\ge 2^{-n/2}$, to obtain
		\begin{equation}\label{eq: apres integration x}\int_{\mathbf R^n} |f(t,x)|\mathrm dx\le C_n\int_0^t \int_{\mathbf R^n}|F(\tau,z)|\mathrm dV_g(z).\end{equation}
		Thus it remains to control $F(\tau,z)=(\Delta_g-\Delta_{g(y)}) H_{g(y)}(\tau,z,y)$, which is a linear combination of first and second space derivatives of $H_{g(y)}(\tau,\cdot,y)$. At small $\tau$, the mass of derivatives of $H_{g(y)}(\tau,\cdot,y)$ concentrates near $y$. Notably, the mass of second derivatives of $H_{g(y)}(\tau,\cdot,y)$ explodes like $\tau^{-1}$ when $\tau\to 0$. However, even though $(\Delta_g-\Delta_{g(y)})$ is a differential operator of order $2$, its term of leading order vanishes at $y$, allowing to tame the singularity of $F(\tau,z)$ at $z=y$. 
		
		Let us put these ideas in practice. First, we recall that the Laplacian $\Delta_g$ is given by
		\[\Delta_g=\frac{1}{\sqrt{|g|}} \mathrm{div}(\sqrt{|g|}g^{-1}\nabla \cdot),\]
		where $\mathrm{div}$ and $\nabla$ denote the standard divergence and gradient operators on $\mathbf R^n$. Then, one can check from the assumptions on $g$ that 
		\[\Delta_g-\Delta_{g(y)} =\sum_{1\le i,j\le n} a_{ij}(z)\frac{\partial^2}{\partial z^i\partial z^j} +\sum_{{1\le i\le n}}b_i(z)\frac{\partial}{\partial z^i},\]
		where the coefficients $a_{ij},b_i$ are smooth functions satisfying estimates $|a_{ij}(z)|\le C_n r^{-1}|z-y|$ and $|b_i(z)|\le C_nr^{-1}$. Thus, for any function $h=h(z)$ we have
		\begin{equation}\label{diff laplaciens}\big|(\Delta_g-\Delta_{g(y)})h(z)\big| \le C_n\Big(r^{-1} |z-y| \cdot\|\mathrm D^2h(z)\|+r^{-1}\|\mathrm Dh(z)\|\Big),\end{equation}
		where $\|\mathrm D^k h(z)\|:=\sup_{|\alpha|=k} |\frac{\partial^{|\alpha|} h(z)}{\partial z^\alpha}|$. Notably, for $h(z)=H_{g(y)}(\tau,z,y)$ we have
		\[\|\mathrm D h(z)\|\le C_n\frac{|z-y|}{\tau}H_{g(y)}(\tau,z,y), \qquad \|\mathrm D^2 h(z)\|\le C_n\Big(\frac{1}{\tau}+\frac{|z-y|^2}{\tau^2}\Big) H_{g(y)}(\tau,z,y).\]
		Consequently, by \eqref{diff laplaciens} we obtain
		\[|F(\tau,z)|\le \frac{C_n}r \left(\frac{|z-y|}{\tau}+\frac{|z-y|^3}{\tau^2}\right) H_{g(y)}(\tau,z,y).\]
		Letting $u(x):=(|x|+|x|^3)\exp(-|x|_{g(y)}^2/4)$ and recalling the expression of $H_{g(y)}$ from \eqref{eq:heatconstant}, the inequality above yields
		\[|F(\tau,z)|\le \frac{C_n}{r\sqrt \tau} \tau^{-\frac n2}u\Big(\frac{z-y}{\sqrt \tau}\Big).\]
		Thus, recalling \eqref{eq: apres integration x} and using the fact that $\sqrt{|g(z)|}\le 2^{\frac n2}$,
		\[\int_{\mathbf R^n}|f(t,x)|\mathrm dx\le \frac{C_n}{r}\int_0^t\frac{\mathrm d\tau}{\sqrt \tau}\int_{\mathbf R^n} u\Big(\frac{z-y}{\sqrt \tau}\Big)\frac{\mathrm dz}{\tau^{n/2}}\le \frac{C_n}{r}\int_0^t \frac{\mathrm d\tau}{\sqrt \tau} \le C_n\frac{\sqrt t}{r}.\]
		In the second inequality we have used the integrability of $u$ on $\mathbf R^n$. This upper bound is rather poor when $t$ is much larger than $r^2$. Nevertheless, since the functions $H_g(t,y,\cdot)\sqrt{|g(\cdot)|}$ and $H_{g(y)}(t,y,\cdot)\sqrt{|g(y)|}$ have mass $1$ and $\sqrt{|g|}\ge 2^{-n/2}$, we can always bound
		\[\int_{\mathbf R^n} |f(t,x)|\mathrm dx\le C_n.\]
		This proves \eqref{withnotationgoal}. The upper bound on $\int_{\mathbf R^n}\big|K_g(x,y)-K_{g(y)}(x,y)\big|\mathrm dx$ follows by integration of \eqref{withnotationgoal} against $\frac{\mathrm dt}{t^{1+s/2}}$ on $\mathbf R_+$.
	\end{proof}

	\begin{proof}[Proof of Proposition \ref{prop:changer noyau}] By the triangular inequality, $\big|K_g(x,y)\sqrt{|g(x)|}-K_{g(y)}(x,y)\sqrt{|g(y)|}\big|$ is smaller than \[ \big|K_g(x,y)-K_{g(y)}(x,y)\big| \sqrt{|g(x)|} +\big|\sqrt{|g(x)|}-\sqrt{|g(y)|}\big| K_{g(y)}(x,y).\]
		From the hypotheses on $g$, we have $\sqrt{|g|}\le 2^{n/2}$ and $\big|\sqrt{|g(x)|}-\sqrt{|g(y)|}\big|\le C_n\min(1,r^{-1}|x-y|)$, thus
		\[\int_{\mathbf R^n} \big|K_g(x,y)\sqrt{|g(x)|}-K_{g(y)}(x,y)\sqrt{|g(y)|}\big|\mathrm dx\]
		is smaller than a constant depending only on $n$, times
		\[\int_{\mathbf R^n}\big|K_g(x,y)-K_{g(y)}(x,y)\big|\mathrm dx+\int_{\mathbf R^n}\min(1,r^{-1}|x-y|)K_{g(y)}(x,y)\mathrm dx.\]
		Both integrals are bounded by $C_{n,s}r^{-s}$. For the first one this is the content of Lemma \ref{effacement singularites}, while the second one can be estimated by using the upper bound $K_{g(y)}(x,y)\le C_{n,s}|x-y|^{-(n+s)}$.  
		
	\end{proof}

	\section{Rephrasing NMC boundedness condition} \label{sec:freeze}  
	In this section we stick to the setting of Theorem \ref{mainthm}, that is $(M,g)=(\mathbf R^n,g)$ with $\frac 12\le g\le 2$ and $r\|\nabla g_{ij}\|_{L^\infty(\mathbf R^n)}\le 1$ for any $i,j\in \{1,\ldots,n\}$. We will go back to the general case only in  Section \ref{fromto}.
	
	As we claimed in the introduction, a measurable subset $E$ of $\mathbf R^n$ has well-defined NMC at $y\in \partial E$ for the metric $g$ whenever $E$ has an interior or exterior tangent ball at $y$. More precisely, we have the following proposition.
 
	\begin{prop}\label{prop: NMC bien définie} Let $g=(g_{ij})$ be a smooth Riemannian metric on $\mathbf R^n$ satisfying $\frac 12\le g\le 2$ and $r\|\nabla g_{ij}\|_{L^\infty(\mathbf R^n)}\le 1$ for any $i,j\in \{1,\ldots,n\}$. Let $E$ be a measurable subset of $\mathbf R^n$ and assume that $E$ has an interior (resp. exterior) tangent ball at $y\in \partial E$. Then, $\mathrm H_s^g[E](y)$ is well-defined, with value in $(-\infty,+\infty]$ (resp. in $[-\infty,+\infty)$).  \end{prop}
	
	When $g$ is constant, this is proved by exploiting cancellations between $E$ and $\mathcal CE$ in the integral \eqref{eq: def NMC} defining $\mathrm{H}_s^g[E]$. This is directly possible because the kernel $K_g(x,y)\mathrm dV_g(x)$ is invariant under the transformation $x\mapsto 2y-x$. This property fails when $g$ is nonconstant. Still, Proposition \ref{prop:changer noyau} shows that we can replace $K_g(x,y)\sqrt{|g(x)|}$ in the definition of $\mathrm{H}_s^g[E](y)$, by the (symmetric) kernel associated to the constant metric $g(y)$, only adding a uniformly bounded term to $\mathrm{H}_s^g[E](y)$. An immediate byproduct of the proof of Proposition \ref{prop: NMC bien définie} will be an estimate on the difference $\mathrm{H}_s^g[E](y)-\mathrm{H}_s^{g(y)}[E](y)$. Here we consider a specific class of metrics on $\mathbf R^n$, but will prove the result under no restriction in Proposition \ref{prop: NMC bien définie bis}.

	\begin{proof}[Proof of Proposition \ref{prop: NMC bien définie}] We start by pointing out that if the principal value \eqref{eq: def NMC} is well-defined, then we can replace metric balls by balls associated to the frozen metric $g(y)$ without altering the definition of $\mathrm{H}_s^g[E](y)$. This doesn't even require exploiting some kind of cancellations in the integrals. Indeed, 
		\begin{equation}\label{peu importe les boules}\left|\int_{\mathbf R^n\backslash B_\varepsilon^g(y)} K_g(x,y)\mathrm dV(x)-\int_{\mathbf R^n\backslash B_\varepsilon^{g(y)}(y)} K_g(x,y)\mathrm dV(x)\right|\le C_{n,s}\int_{B_\varepsilon^g(y)\triangle B_\varepsilon^{g(y)}(y)} \frac{\mathrm dx}{|x-y|^{n+s}}.\end{equation}
		Since $g$ is smooth, we have $|B_\varepsilon^g(y)\triangle B_\varepsilon^{g(y)}(y)|\le C \varepsilon^{n+1}$ as $\varepsilon\to 0$. Combined with the fact that $\mathrm{dist}(B_\varepsilon^g(y)\triangle B_\varepsilon^{g(y)}(y),y)\ge \frac 12\varepsilon$, this shows that \eqref{peu importe les boules} is bounded by a constant times $\varepsilon^{1-s}$. Since $s<1$, this quantity goes to $0$ as $\varepsilon\to 0$.
		
		We turn to the proof of the Proposition. Without loss of generality, assume that $E$ has an interior tangent ball at $y\in \partial E$. For $\varepsilon>0$, write
		\[\int_{\mathbf R^n\backslash B^{g(y)}_\varepsilon(y)} \big(\chi_E-\mathcal \chi_{\mathcal C E}\big)(x)K_g(x,y)\mathrm dV(x)=I_1+I_2,\]
		where we set
		\[I_1= \int_{\mathbf R^n\backslash B^{g(y)}_\varepsilon(y)} \big(\chi_E-\mathcal \chi_{\mathcal C E}\big)(x)K_{g(y)}(x,y)\sqrt{|g(y)|} \mathrm dx,\]
		and
		\[I_2= \int_{\mathbf R^n\backslash B^{g(y)}_\varepsilon(y)} (\chi_E-\mathcal \chi_{\mathcal C E})(x)\big(K_g(x,y)\sqrt{|g(x)|}-K_{g(y)}(x,y) \sqrt{|g(y)|}\big) \mathrm dx.\]
		According to Proposition \ref{prop:changer noyau}, the integral $I_2$ is absolutely convergent uniformly with respect to $\varepsilon$, and is bounded by a constant (depending only on $n$ and $s$) times $r^{-s}$. Let us now handle $I_1$. Let $\mathbf B^+\subset E$ be a Euclidean ball tangent to $\partial E$ at $y$, and write $\chi_E-\chi_{\mathcal C E}=\chi_{\mathbf B^+}-\chi_{\mathcal C \mathbf B^+} +2\chi_{E\backslash \mathbf B^+}$, so that we can split $I_1=I_3+I_4$  with
		\[I_3=\sqrt{|g(y)|}\int_{\mathbf R^n\backslash B^{g(y)}_\varepsilon(y)} \big(\chi_{\mathbf B^+}-\mathcal \chi_{\mathcal C \mathbf B^+}\big)(x)K_{g(y)}(x,y) \mathrm dx\]
		and 
		\[I_4=2\sqrt{|g(y)|}\int_{\mathbf R^n\backslash B^{g(y)}_\varepsilon(y)} \chi_{E\backslash \mathbf B^+}(x) K_{g(y)}(x,y) \mathrm dx.\]
		We can exploit cancellations between $\mathbf B^+$ and $\mathcal C \mathbf B^+$ thanks to the facts that $K_{g(y)}(y+z,y)=K_{g(y)}(y-z,y)$ and that the ball $B_\varepsilon^{g(y)}(y)$ is invariant by the transformation $y+z\mapsto y-z$. Indeed, if we let $\mathbf B^-$ denote the Euclidean ball obtained by flipping $\mathbf B^+$ along the tangent to $\mathbf B^+$ at $y$, then we have
		\[I_3=-\sqrt{|g(y)|}\int_{\mathbf R^n\backslash B^{g(y)}_\varepsilon(y)} \chi_{\mathcal C(\mathbf B^+\cup \mathbf B^-)}(x)K_{g(y)}(x,y) \mathrm dx.\]
		The integrand is absolutely convergent on $\mathbf R^n$, as one can check from the smoothness of $\mathbf B^+$ and the upper bound $K_{g(y)}(x,y)\le C_{n,s} |x-y|^{-(n+s)}$.
		Finally, $I_4$ is the integral of a positive function over $\mathbf R^n\backslash B_\varepsilon^{g(y)}(y)$, thus it converges to some value in $[0,+\infty]$ as $\varepsilon\to 0$. This concludes the proof. We point out that $I_1$ just converges to $\mathrm{H}_s^{g(y)}[E](y)$ as $\varepsilon\to 0$.
	\end{proof}
	
	\begin{rem} \label{Remarque shrinking balls} The proof shows that whenever $E$ has a tangent ball at $y\in \partial E$, we can replace the shrinking balls $B_\varepsilon^g(y)$ in the definition of $\mathrm H_s[E](y)$ by any family of shrinking sets $(U_\varepsilon)$ that contain $y$, and that are invariant under the transformation $x\mapsto 2y-x$. Even more generally, for any smooth diffeomorphism $\varphi$ defined on a neighborhood of $0$ and mapping $0$ to $y$, we may take $U_\varepsilon=\varphi(B_\varepsilon(0))$ without changing the value of $H_s[E](y)$. Indeed, if $\mathcal T_y$ denotes the symmetry $x\mapsto 2y-x$ then $W_\varepsilon:= U_\varepsilon\cup\mathcal T_yU_\varepsilon$ satisfies $y\in W_\varepsilon$ and $\mathcal T_yW_\varepsilon=W_\varepsilon$, thus $(W_\varepsilon)$ is an admissible family of shrinking sets, then we observe that $|W_\varepsilon\triangle U_\varepsilon|= |\mathcal T_y U_\varepsilon\backslash U_\varepsilon|\le C_{n,\varphi} \varepsilon^{n+1}$ and $\mathrm{dist}(W_\varepsilon\triangle U_\varepsilon)\ge c \varepsilon$ for some constant $c>0$ depending on $\varphi$, thus we can argue as in \eqref{peu importe les boules}. In particular, we can work with Euclidean balls in the following. 
	\end{rem}
	By carefully looking at the proof of Proposition \ref{prop: NMC bien définie}, we see that sets with finite NMC at a point $y$ for the metric $g$ have finite NMC at $y$ for the constant metric $g(y)$ (in the viscosity sense).
	
	\begin{prop}\label{approxeulerlagrange}Let $g$ be a smooth Riemannian metric on $\mathbf R^n$ satisfying $\frac 12\le g\le 2$ and  $r \|\nabla g_{ij}\|_{L^\infty(\mathbf R^n)}\le 1$ for any $i,j\in \{1,\ldots,n\}$. Consider a measurable subset $E\subset \mathbf R^n$, and assume that $E$ has NMC bounded by $C_0r^{-s}$ in the viscosity sense in $\Omega\subset \mathbf R^n$ (for the metric $g$). Then, whenever $E$ has an interior (resp. exterior) tangent ball at $y\in \partial E\cap \Omega$, the NMC of $E$ at $y$ for the constant metric $g(y)$ is well-defined, and we have $\mathrm H_s^{g(y)}[E](y)\le (C_0+C_{n,s})r^{-s}$ (resp. $\mathrm H_s^{g(y)}[E](y)\ge -(C_0+C_{n,s})r^{-s}$.)
	\end{prop}

	\section{Improvement of flatness for sets of bounded nonlocal mean curvature}\label{Improvement of flatness}

	In this section, we show an improvement of flatness theorem for subsets $E\subset (\mathbf R^n,g)$ of bounded nonlocal mean curvature in the viscosity sense for the metric $g$, when $g$ satisfies the hypotheses of Theorem \ref{mainthm}. The proof for subsets of $\mathbf R^n$ with zero NMC for the Euclidean metric is done in \cite[\S6]{CRS} and we will follow it closely, although we need to work out again the arguments since we are dealing with nonconstant metrics and only assume a bound on the nonlocal mean curvature. In practice, though, Proposition \ref{approxeulerlagrange} will allow us to deal with kernels associated to constant metrics.
	
	 \begin{thm}[Improvement of flatness] \label{improvement} There exists $k_0\in \mathbf N$ depending on $C_0,n,s$ and $\alpha$ such that the following holds. Let $g=(g_{ij})$ be a smooth Riemannian metric on $\mathbf R^n$ satisfying $\frac 12\le g\le 2$ and  $r\|\nabla g_{ij}\|_{L^\infty(\mathbf R^n)}\le 1$ for any $i,j\in \{1,\ldots,n\}$. Let $E$ be a measurable subset of $\mathbf R^n$ with $NMC$ bounded by $C_0r^{-s}$ in the viscosity sense in $B_r(0)$ (for the metric $g$). Moreover, assume $0\in \partial E$ and 
			\[\{x\cdot \nu_l\le -r 2^{-l(\alpha+1)}\}\subset  E\subset \{x\cdot \nu_l\le r 2^{-l(\alpha+1)}\}~\text{in $B_{r2^{-l}}(0)$},\]
			for some family of unit vectors $(\nu_l)_{0\le l\le k}$, with $k\ge k_0$. Then there exists a unit vector $\nu_{k+1}$ such that
			\[\{x\cdot \nu_{k+1}\le -r 2^{-(k+1)(\alpha+1)}\}\subset  E\subset \{x\cdot \nu_{k+1}\le r2^{-(k+1)(\alpha+1)}\}~\text{in $B_{r2^{-(k+1)}}(0)$}.\]
	\end{thm}
	\begin{figure}[h]
		\frame{\includegraphics[width=13cm]{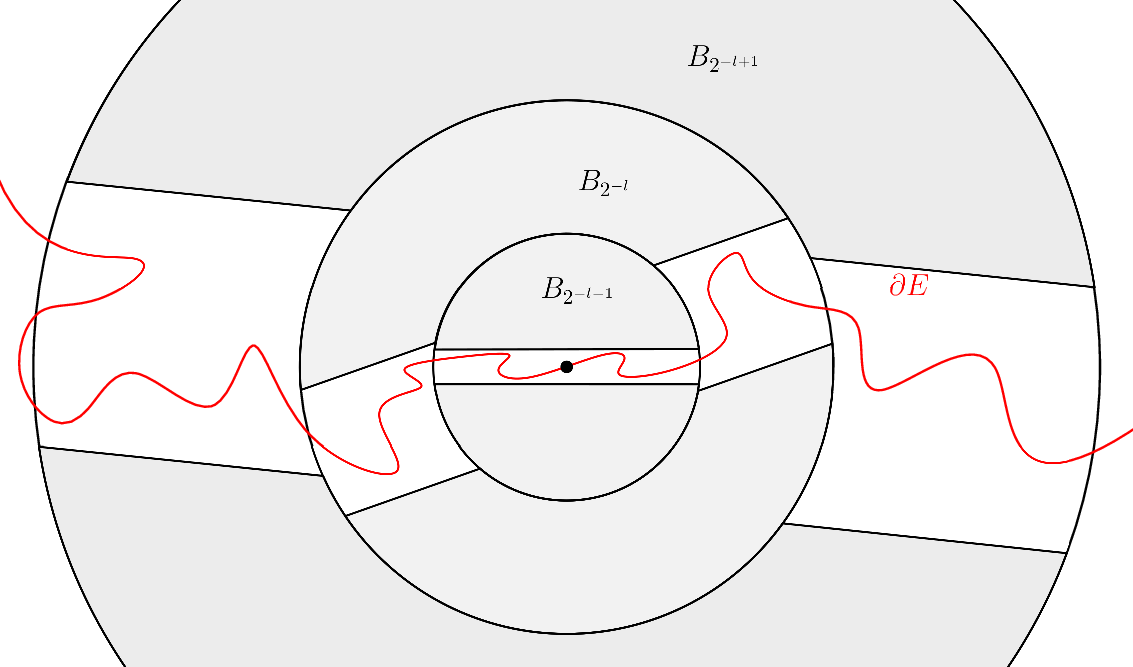}}
		\caption{The trapping hypothesis on $E$, for $r=1$. Observe that normal vectors to the cylinders cannot be too far apart from each other.}
		\label{flatcyl}
		\centering\end{figure}
	
	Let us give a sketch of the proof of the theorem before diving into the details. From now on we assume $r=1$, as the general result follows by performing a rescaling.
	
	\begin{proof}[Sketch of the proof] 
		The proof works by contradiction. Assume that we have a sequence of metrics $(g_k)$ on $\mathbf R^n$ satisfying the assumptions of Theorem \ref{improvement} with $r=1$, and a sequence of measurable sets $\{E_k\}_{k\ge 1}$ with $0\in \partial E_k$, such that $E_k$ has NMC bounded by $C_0$ (in the viscosity sense) for the metric $g_k$ in $B_1(0)$, and such that the inclusions in the statement of Theorem \ref{improvement} hold for $E_k$, for $j=0,\ldots,k$, say
		\begin{equation}\label{introset} \{x\cdot \nu_j^k\le -2^{-j(\alpha+1)}\}\cap B_{2^{-j}}(0)\subset E_k\cap B_{2^{-j}}(0)\subset \{x\cdot \nu_j^k\le 2^{-j(\alpha+1)}\}\end{equation} 
		for some family of unit vectors $(\nu_j^k)$, but the conclusion fails. Up to rotating the sets $E_k$, we can assume $\nu_k^k=e_n$ for all $k$. This requires in turn to pullback the metric $g_k$ by a rotation, but this causes no harm because the class of metrics that we are studying is invariant under such transformations. 
	We dilate the sets $E_k$ by a factor $2^{k}$, then stretch them by a factor $\frac{1}{a_k}=2^{k\alpha}$ in the normal direction $e_n$. It yields a new sequence $\{ \widetilde E_k^*\}$. In Proposition \ref{prop:harnack}, we show that sets with bounded NMC whose boundaries are trapped in a flat cylinder cannot oscillate too much in the direction normal to the cylinder. This implies that up to a subsequence, $\{\widetilde E_k^*\}$ converges, uniformly on compact subsets of $\mathbf R^n$, to a half-space whose boundary passes through the origin. This fact contradicts the assumption that the sets $E_k$ fail to satisfy the conclusion of Theorem \ref{improvement}.
\end{proof}

\subsection{Harnack inequality} \label{sec: harnack} The first step towards the result is to prove a Harnack-type inequality, allowing to control oscillations of sets of bounded NMC whose boundaries are trapped in a flat cylinder. As we shall see, this result can be iterated when the cylinder is sufficiently flat.

\begin{prop}[Harnack inequality]\label{prop:harnack} There exists $\delta\in (0,1)$ and $k_1\in \mathbf N$ depending on $C_0,n,s$ and $\alpha$, such that the following holds.
	Let $g=(g_{ij})$ be a smooth Riemannian metric on $\mathbf R^n$, satisfying $\frac 12\le g\le 2$ and  $\|\nabla g_{ij}\|_{L^\infty(\mathbf R^n)}\le 1$ for any $i,j\in \{1,\ldots,n\}$.  Let $E$ be a measurable subset of $\mathbf R^n$, with NMC bounded by $C_0$ in the viscosity sense in $B_1(0)$, for the metric $g$. Moreover, assume that there is some $k\ge k_1$ such that
	\begin{equation}\label{eq:assumption inclusion harnack}\{x\cdot \nu_l\le -2^{-l(1+\alpha)}\}\subset  E\subset \{x\cdot \nu_l\le 2^{-l(1+\alpha)}\}~ \text{in $B_{2^{-l}}(0)$},\end{equation}
	for some family of unit vectors $(\nu_l)_{0\le l\le k}$, with $\nu_k=e_n$. Then, we have either
	\[\{|x'|\le 2^{-k}\delta\}\times \{-2^{-k(1+\alpha)}\le x^n\le 2^{-k(1+\alpha)}(-1+\delta^2)\} \subset E\]
	or
	\[\{|x'|\le 2^{-k}\delta\}\times \{2^{-k(1+\alpha)}(1-\delta^2)\le x^n\le 2^{-k(1+\alpha)}\}\subset \mathcal CE.\]
\end{prop}	


\begin{proof} We follow the proof of \cite{CRS}. The difference is that the assumption that $E$ has zero NMC is replaced by a NMC bound by $C_0$ in the viscosity sense in $B_{1}(0)$ (for a nonconstant metric). 
	
	By Proposition \ref{approxeulerlagrange}, there is a constant $C_1$ depending only on $C_0,n,s$ such that whenever $E$ has an interior tangent ball at $y\in\partial E\cap B_1(0)$, we have
	\begin{equation} \label{EulerLagrange}\mathrm{p.v.}\int_{\mathbf{R}^n} \big(\chi_E(x)-\chi_{\mathcal C E}(x)\big)K_{g(y)}(x,y) \mathrm dx\le C_1,\end{equation}
	with an analogous statement whenever $E$ has an exterior tangent ball at $y\in \partial E\cap B_1(0)$. Now that we are dealing with the kernel $K_{g(y)}(x,y)=\alpha_{n,s} |x-y|_{g(y)}^{-(n+s)}$, the rest of the proof mimics that of the Euclidean case.
	
	\textit{Step 1. Estimating the nonlocal contribution.} Assume $y\in B_{2^{-k-1}}(0)$. We wish to estimate the nonlocal contribution
	\[\left|\int_{\mathcal C B_{2^{-k-1}}(y)} \big(\chi_E(x)-\chi_{\mathcal C E}(x)\big)K_{g(y)}(x,y) \mathrm dx\right|.\]
	Here the integral is convergent in the usual sense. First, we start by controlling the tail, using Proposition \ref{FBmetric}:
	\begin{equation}\label{tail}\left|\int_{\mathcal C B_{1/2}(y)}\big(\chi_E(x)-\chi_{\mathcal C E}(x)\big)K_{g(y)}(x,y) \mathrm dx\right|\le 2\int_{\mathcal C B_{1/2}(y)} K_{g(y)}(x,y) \mathrm dx\le C_{n,s}.\end{equation}
	Now we aim to bound the contribution of the dyadic annuli $B_{2^{-l}}(y)\setminus B_{2^{-l-1}}(y)$, that we denote by $I_l$. Here we use \eqref{eq:assumption inclusion harnack} to get cancellations in the integrals. Since $y\in B_{2^{-k-1}}(0)$, we have $B_{2^{-l}}(y)\subset B_{2^{-(l-1)}}(0)$ for $l=1,\ldots,k+1$. Recalling \eqref{eq:assumption inclusion harnack} we obtain that for $l=1,\ldots,k+1$ we have
	\[\{(x-y)\cdot \nu_{l-1}\le -2^{1-(l-1)(1+\alpha)}\}\subset E\subset \{(x-y)\cdot \nu_{l-1}\le 2^{1-(l-1)(1+\alpha)}\}~\text{in $B_{2^{-l}}(y)$}.\]
	By using the symmetry property $K_{g(y)}(y+z,y)=K_{g(y)}(y-z,y)$ and the upper bound on $K_{g(y)}(x,y)$ we get
	\begin{equation}\label{dyadic} |I_l| \le C_{n,s} \displaystyle\int_{B_{2^{-l}}(y)\setminus B_{2^{-l-1}}(y)} \mathbf{1}_{\{|(x-y)\cdot \nu_{l-1}|\le 2^{2+\alpha}2^{-l(1+\alpha)} \}} \frac{\mathrm dx}{|x-y|^{n+s}} \le C_{n,s}2^{l(s-\alpha)},\end{equation}
	for $l=1,\ldots,k+1$. Summing inequalities (\ref{dyadic}) for $l=1,\ldots,k+1$ we obtain
	\begin{equation}\label{unionanneaux}\left|\int_{B_{1/2}(y)\setminus B_{2^{-k-1}}(y)}(\chi_{E}-\chi_{\mathcal C E})(x) K_{g(y)}(x,y) \mathrm dx\right|\le C_{n,s}2^{k(s-\alpha)}.\end{equation}
	Finally, as $s-\alpha>0$, combining (\ref{tail}) and (\ref{unionanneaux}), the nonlocal contribution is bounded by
	\begin{equation}\label{eq:nonlocal contrib at last}\left|\int_{\mathbf{R}^n\setminus B_{2^{-k-1}}(y)}(\chi_{E}-\chi_{\mathcal C E})(x) K_{g(y)}(x,y) \mathrm dx\right|\le C_{n,s}2^{k(s-\alpha)}.\end{equation}

	\textit{Step 2. Local estimates.} Let $a=2^{-k\alpha}$, and assume that $k$ and $\delta$ are chosen to ensure $a\le \delta$. Let us work with the rescaled set $\widetilde E:=2^kE$. By the hypothesis on $E$, we have 
	\[\{x^n<-a\}\cap B_1\subset \widetilde E \cap B_1\subset \{x^n<a\}.\] We also assume that $\widetilde E $ contains more than half of the measure of the cylinder $D=\{|x'|\le \delta\}\times \{|x^n|\le a\}$, that is,
	\begin{equation}\label{eq: moitié mesure} |\widetilde E \cap D|\ge \frac 12 |D|.\end{equation}
	If not, we can just work with the complement $\mathcal C \widetilde E $. Let us show that it implies \begin{equation}\label{eq: claim inclusion} \{x^n\le (-1+\delta^2)a\}\cap D\subset \widetilde E, \end{equation}
	for an adequate choice of $\delta$. If \eqref{eq: claim inclusion} does not hold, take some point $z\in \{x^n\le (-1+\delta^2)a\}\cap D$ with $ z\notin \widetilde E$. We slide by below the parabola $\{x^n=-\frac{a}{2}|x'|^2\}$, until we touch $\partial \widetilde E$ at a point $\tilde y=2^k y\in B_1$, see Figure \ref{figsliding'}. In particular, $E$ has an interior tangent ball at $y\in \partial E$ so we can use \eqref{EulerLagrange}. \begin{figure}[h]
		\frame{\includegraphics[width=13cm]{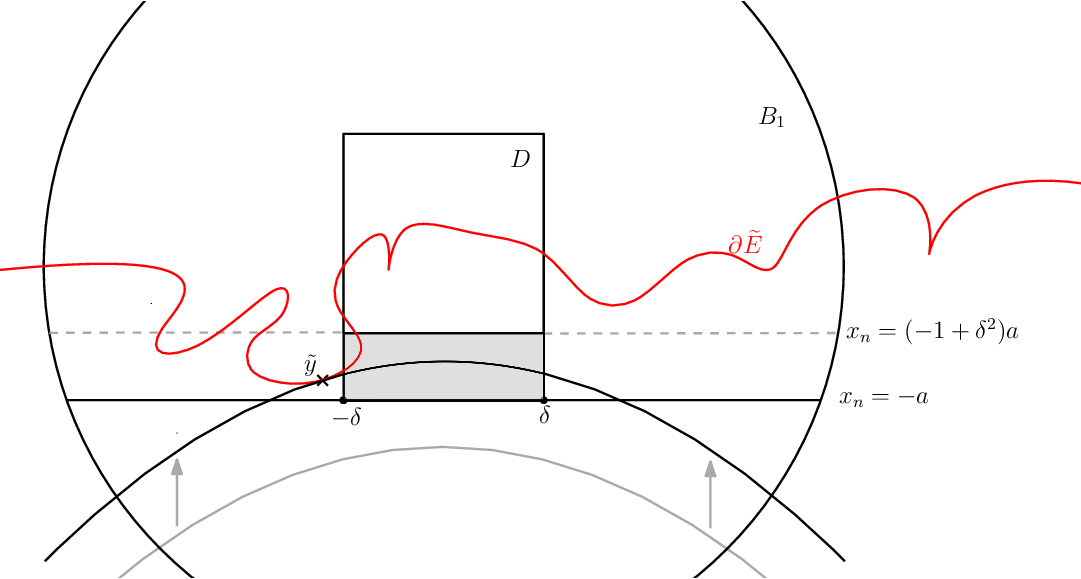}}
		\caption{Sliding a parabola until touching $\partial \widetilde E$. The shaded area $\{|x'|\le \delta\}\times \{-a<x^n<(-1+\delta^2)a\}$ is entirely contained under the parabola when $t\ge a(-1+\frac 32\delta^2)$. Notice that the contact point $\tilde y$ need not belong to the shaded area, but $|\tilde y'|\le 2\delta$ and $|\tilde y^n+a|\le 2a\delta^2$. We stress out that $\partial \widetilde E$ need not be a graph.}
		\label{figsliding'}
		\centering\end{figure}
	Denote by $t\mapsto \{(x',-a/2|x'|^2 +t)\}$ the sliding graphs, and let $t_0$ denote the first time at which the graph hits $\partial \widetilde E$ in $B_1$. We claim that $t_0\le a(-1+3/2\delta^2)$. If not, there is some $t>a(-1+3/2\delta^2)$ such that the subgraph of $x'\mapsto -a/2|x'|^2 +t$ doesn't intersect $\mathcal C\widetilde E$ in $D$. This is absurd because the point $z$ chosen above belongs to this subgraph, since
	\[z^n\le (-1+\delta^2) a\le -\frac a2|z'|^2+t.\]
	
	Notice that $-a\le \tilde y^n\le t_0$, giving $|\tilde y_n+a|\le 2a\delta^2$. Since $\tilde y^n=-\frac a2 |\tilde y'|^2+t_0$ we deduce $|\tilde y'|^2\le 3\delta^2$. Altogether we have	
	\begin{equation}\label{controly}|\tilde y'|\le 2\delta, \ \ \ \ \ |\tilde y^n+a|\le 2a\delta ^2.\end{equation}
	Note that in particular $y\in B_{2^{-k-1}}(0)$, provided that $\delta$ is small enough, which allows using \eqref{eq:nonlocal contrib at last}.
	
	Denote by $\widetilde P$ the subgraph of the parabola touching $\partial \widetilde E$ at $\tilde y$, and let $ P:=2^{-k}\widetilde P$ denote the subgraph scaled back by a factor $2^{-k}$. Then, one has
	\[\begin{array}{ll}\displaystyle \mathrm{p.v.} \int_{B_{2^{-k-1}}(y)} (\chi_{E}-\chi_{\mathcal C E})(x) K_{g(y)}(x,y) \mathrm dx &=\displaystyle \mathrm{p.v.} \int_{B_{2^{-k-1}}(y)} (\chi_P -\chi_{\mathcal C P})(x)K_{g(y)}(x,y)\mathrm dx\\ &+2\displaystyle\int_{B_{2^{-k-1}}(y)} \chi_{E\setminus P}(x)K_{g(y)}(x,y) \mathrm dx\\ &=: I_3+I_4,\end{array}\]
	where $I_4$ is the integral of a positive measurable function, thus $I_4$ belongs to $ [0,+\infty]$. 
	
	\textit{Step 3. Lower bound on $I_3$.} Denote by $\nu$ the normal vector to the parabola at $\widetilde y$. Then there is some constant $C$ depending only on $n$ such that
	\[\partial \widetilde P\cap B_\rho(\tilde y)\subset \{x, \ |(x-\tilde y)\cdot \nu|\le C a\rho^2\}, \qquad \text{for all $\rho\in [0,1/2]$}.\]
	Thanks to the symmetry property $K_{g(y)}(y,y+z)=K_{g(y)}(y,y-z)$ we get some cancellations between $P$ and $\mathcal CP$ when computing $I_3$, and obtain a lower bound
	\begin{equation}\label{I3} I_3\ge -C_{n,s}2^{ks}\int_0^{1/2}\frac{a\rho^{n}}{\rho^{n+s}} \mathrm d\rho\ge -C_{n,s} 2^{k(s-\alpha)}.\end{equation}
	
	\textit{Step 4. Lower bound on $I_4$.} Since $t_0+a\le \frac 32 a\delta^2$ we have $|\widetilde P\cap D|\le \delta^2 |D|$. Together with \eqref{eq: moitié mesure} this implies $|(\widetilde E\setminus \widetilde P)\cap D|\ge (1/2-\delta^2)|D|$. Also, by $(\ref{controly})$ and the fact that $a\le \delta$, we get, for any $x\in D$,
	\[|\tilde y-x|=\sqrt{|\tilde y'-x'|^2+|\tilde y^n-x^n|^2}\le \sqrt{(3\delta)^2 +(2a)^2}\le 4\delta.\]
	Using the lower bound $K_{g(y)}(x,y)\ge c_{n,s}|x-y|^{-(n+s)}$ and the fact that the volume of $D$ is a multiple of $a\delta^{n-1}$ by a dimensional constant, we infer
	\begin{equation} \label{I4}I_4\ge c_{n,s}2^{ks}(1/2-\delta^2)\frac{a\delta^{n-1}}{(4\delta)^{n+s}}\ge c_{n,s}\delta^{-1-s}2^{k(s-\alpha)},\end{equation}
	provided that $a\le \delta <\frac 12$. Combining (\ref{I3}) and (\ref{I4}) we get
	\[\mathrm{p.v.}\int_{B_{2^{-k-1}}(y)} \big(\chi_{E}(x)-\chi_{\mathcal C E}(x)\big) K_{g(y)}(x,y) \mathrm dx \ge 2^{k(s-\alpha)}(-C_{n,s}+c_{n,s}\delta^{-1-s}),\]
	given $a\le \delta$. Finally, combining this lower bound with \eqref{eq:nonlocal contrib at last} we obtain, up to increasing $C_{n,s}$,
	\[\begin{array}{ll}\mathrm{p.v.}\displaystyle\int_{\mathbf{R}^n} \big(\chi_E(x)-\chi_{\mathcal C E}(x)\big)K_{g(y)}(x,y) \mathrm dx& \ge  2^{k(s-\alpha)}(-C_{n,s}+c_{n,s}\delta^{-1-s})\\ &\ge 2^{k(s-\alpha)}.\end{array} \]
	where the last inequality holds for $\delta$ small enough, depending on $C_0$,$n$,$s$,$\alpha$, and $k\ge k_1=k_1(\delta)$ is large enough to ensure $a\le \delta$. We reach a contradiction to the NMC boundedness (\ref{EulerLagrange}) when $k$ is large, thus proving \eqref{eq: claim inclusion}, which is the content of the Proposition.
\end{proof}


\subsection{Iterating Harnack inequality} Consider a subset $E\subset \mathbf R^n$ satisfying the assumptions of Proposition \ref{prop:harnack}. When $k\gg k_1$ we can apply Harnack inequality multiple times to control the oscillations of $\partial E$ in the normal direction. This procedure is described in the following. We point out that it can already be found in great detail in the literature when $g$ is the Euclidean metric, see e.g. \cite[Chapter 5]{Lombardini}. Our setup is slightly different, since we are dealing with sets of possibly nonzero NMC for a nonconstant metric (instead of sets of vanishing NMC for the Euclidean metric). In particular, we have to check that NMC remains adequately bounded when iterating Harnack inequality.

As before, we introduce the rescaled set $\widetilde E=2^k E$. Letting $e_j:=\nu_{k-j}$, for $j=0,\ldots,k$ we have
\[\{x\cdot e_j\le -a 2^{j(1+\alpha)}\}\subset \widetilde E\subset \{x\cdot e_j\le a 2^{j(1+\alpha)}\}~\text{in $B_{2^{j}}(0)$}.\]
We apply Harnack inequality to the set $E$, and assume without loss of generality that the first conclusion of Proposition \ref{prop:harnack} holds, namely that $\{x^n\le (-1+\delta^2)a\}\subset \widetilde E$ in $\{|x'|\le \delta\}\times \{ |x^n|\le a\}$. We translate $\widetilde E$ downwards by a distance $t=\frac 12a\delta^2$, \emph{i.e.} consider $\widetilde E_t=\widetilde E-t e_n$. Then, 
\begin{equation}\label{eq: Harnack translaté} \Big\{x^n\le - a\big(1-\frac{\delta^2}{2}\big)\Big\}\subset \widetilde E_t\subset \Big\{x^n\le  a\big(1-\frac{\delta^2}{2}\big)\Big\}~\text{in $B_\delta(0)$}.\end{equation}
Now, let $\widetilde F:=\frac 1\delta \widetilde E_t$, and define $k'$ by
\begin{equation} \label{defk'} k'=\max\Big \{j\ge 0, \ 2^{-j\alpha}\ge 4 \frac{1-\delta^2/2}{\delta}2^{-k\alpha}\Big\}.\end{equation} 
Finally, let $F:=2^{-k'}\widetilde F$. Then $F$ satisfies the assumptions of Proposition \ref{prop:harnack} with $k$ replaced by $k'$. More precisely, we have:
\begin{prop} There is an explicit metric $h$ on $\mathbf R^n$ satisfying $\frac 12\le h\le 2$ and $\|\nabla h_{ij}\|_{L^\infty(\mathbf R^n)}\le 1$, such that $F$ has NMC bounded by $C_0$ in $B_1(0)$ in the viscosity sense, for the metric $h$. Also, there are unit vectors $(\nu_j')_{0\le j\le k'}$, with $\nu_{k'}'=e_n$, such that
	\[\{x\cdot \nu_j'\le -2^{-j(1+\alpha)}\}\subset  F\subset \{x\cdot \nu_j'\le 2^{-j(1+\alpha)}\} \ \text{in $B_{2^{-j}}(0)$},\]
 for any $j\in \{0,\ldots,k'\}$.\end{prop}

\begin{proof} \emph{1. Inclusion in flat cylinders.} Here the argument is the same as in the Euclidean setting (see \cite[Chapter 5]{Lombardini}), since no metric is involved. Setting $a':=2^{-k'\alpha}$, \eqref{defk'},\eqref{eq: Harnack translaté} imply
	\begin{equation}\label{inclusion 0}\{x^n\le -a'\}\subset \widetilde F\subset \{x^n\le a'\} \ \text{in $B_1(0)$.}\end{equation}
	Since $t\le a$, for $j\in \{0,\ldots,k-1\}$ we have,
	\begin{equation}\label{eq: translaté nouvelles inclusions} \{x\cdot e_{j+1}\le -2^{2+\alpha}a 2^{j(1+\alpha)}\}\subset \widetilde E_t\subset \{x\cdot e_{j+1}\le 2^{2+\alpha}a 2^{j(1+\alpha)}\} \ \text{in $B_{2^j}(0)$}.\end{equation}
	We can always take $\delta$ of the form $\delta=2^{-M_0}$ for some integer $M_0$. Notice that as long as $0\le j-M_0\le k-1$, we have
	\[\widetilde F\cap B_{2^j}(0)\subset 2^{M_0}\big(\widetilde E_t\cap B_{2^{j-M_0}}(0)\big).\]
	Recalling \eqref{eq: translaté nouvelles inclusions}, we set $e_j'=e_{j+1-M_0}$ for $j\in\{M_0,\ldots,M_0+k-1\}$, and take $M_0$ large enough to ensure $M_0\alpha>2+\alpha$. Since $a\le a'$, we obtain for $j\in \{M_0,\ldots,k+M_0-1\}$,
	\begin{equation}\label{inclusions M0 to k} \{x\cdot e_{j}'\le -2^{j(1+\alpha)} a' \}\subset \widetilde F\subset \{x\cdot e_{j}'\le 2^{j(1+\alpha)} a' \} \ \text{in $B_{2^j}(0)$}.\end{equation}
	Now, if $j\in \{1,\ldots,M_0-1\}$ we have $B_{2^{j-M_0}}(0)\subset B_1(0)$ thus
	\[\widetilde F \cap B_{2^{j}}(0)\subset 2^{M_0}(\widetilde E_t\cap B_{1}(0)).\]
	Using \eqref{eq: translaté nouvelles inclusions} with $j=0$ together with the fact that $2^{1+M_0}a=\frac{2a}{\delta}\le a'$ (it follows from the definition of $k'$, assuming $\delta<1$), we can set $e_j'=e_1$ to get, for $j\in \{1,\ldots,M_0-1\}$,
	\begin{equation}\label{inclusions 1 to M0-1} \{x\cdot e_j'\le -2^{j(1+\alpha)} a'\}\subset \widetilde F \subset \{x\cdot e_j'\le 2^{j(1+\alpha)} a'\} ~ \text{in $B_{2^j}(0)$}.\end{equation}
	Since $k'\le k+M_0-1$, by combining \eqref{inclusion 0}, \eqref{inclusions M0 to k}, \eqref{inclusions 1 to M0-1} we obtain
	\[\{x\cdot e_j'\le -2^{j(1+\alpha)}a'\}\subset \widetilde F\subset \{x\cdot e_j'\le 2^{j(1+\alpha)}a'\} \ \text{in $B_{2^j}(0)$},\]
	for any $j\in\{0,\ldots, k'\}$. Setting $\nu_j'=e_{k'-j}'$ gives the claimed inclusions.

	\emph{2. The set $F$ has bounded NMC in $B_1(0)$.} Unraveling the definition of $F$, we see that $E=\varphi(F)$, with $\varphi(x)=2^{-k}(2^{k'}\delta x-t)$.
	Then by Remark \ref{rem: scaling H_s}, whenever $E$ has a tangent ball at $x\in \partial E$ we have
	\[\mathrm{H}_s^{\varphi^*g}[F](\varphi^{-1}(x))=\mathrm{H}_s^g[E](x).\]
	Now just check that $(\varphi^* g)(x)=\lambda^{-2} g(\varphi(x))$, with $\lambda=2^{k-k'}\delta^{-1}>1$. Thus, by setting $h(x)=\lambda^2 \varphi^*g(x)=g\big(2^{-k}(2^{k'}\delta x-t)\big)$ we have by Remark \ref{rem: scaling H_s} again, whenever $F$ has a tangent ball at $y\in \partial F$,
	\[\mathrm{H}_s^{h}[F](y)=\lambda^{-s}\mathrm{H}_s^{\varphi^* g}[F](y)=\lambda^{-s}\mathrm{H}_s^g[E](\varphi(y)).\]
	One can check that $B_1(0)\subset \varphi^{-1}(B_1(0))$, thus $F$ has NMC bounded by $C_0$ in the viscosity sense in $B_1(0)$ for the metric $h$. To conclude the proof, just check that $\frac 12\le h\le 2$ and $\|\nabla h_{ij}\|_{L^\infty(\mathbf R^n)}\le \lambda^{-1}\le 1$. \end{proof} 

We can apply Harnack inequality to $F$ provided that $k'\ge k_1$. Let us now explain the iteration argument. Let $\theta:=1-\frac{\delta^2}{2}$. Define $(k^{(j)})$ inductively by setting $k^{(0)}=k$ and
\[k^{(j+1)}=\max\Big\{m\ge 0, \ 2^{-m\alpha}\ge \frac{4\theta}{\delta} 2^{-k^{(j)}\alpha}\Big\}.\] 
We can iterate Harnack inequality as long as $k^{(j)}\ge k_1$, where $k_1$ is the threshold after which we fail to check the assumptions of Proposition \ref{prop:harnack}. After iterating Harnack inequality $j$ times, we get that $\partial \widetilde E$ is trapped in a cylinder of height $\theta^j a$ in $B_{\delta^j}'(0)\times [-a,a]$, whose center isn't necessarily $0$, but lies on $\{x'=0\}$. In other words, at each step there is some $z_j\in [-a,a]$ such that in $ \{|x'|\le \delta^j\}\times \{|x^n|\le a\}$ we have the inclusions
\[\{x^n\le z_j\}\subset \widetilde E\subset \{x^n\le z_j+\theta^j a\}.\]

Denote by $j(a)$ the maximum value of $j$ for which $k^{(j)}\ge k_1$. Then, as $k\to +\infty$ (or equivalently $a\to 0$) we have
\[j(a)\sim \frac{k}{\lceil \alpha^{-1} \log(4\theta/\delta) \rceil},\]
where $\lceil\cdot \rceil$ denotes the ceiling function. 

\begin{prop}\label{approximation partialE} Let $\gamma\in (0,1)$ be defined by $\theta=\delta^\gamma$. Let $\{E_k\}$ be the sequence of sets introduced in (\ref{introset}). Denote by $B_{1/2}^*$ the ball $B_{1/2}$ stretched by a factor $\frac{1}{a_k}=2^{k\alpha}$ in the $x^n$ direction. Then, for any $x,y\in \partial \widetilde E_k^*\cap B_{1/2}^*$ we have
	\[|y^n-x^n|\le C\max(b_k^\gamma,|y'-x'|^\gamma),\]
	with $b_k\to 0$ as $k\to +\infty$, and $C=C_\delta$.\end{prop}

\begin{proof}We start by proving the result for points $y$ lying on $\{y'=0\}$. Accordingly, let $ y^n\in [-1,1]$ be such that $(0, y^n)\in \partial \widetilde E\cap B_1(0)$. Take $x=(x',x^n)\in \partial \widetilde E\cap B_1(0)$. Assume $\delta^{j+1}\le |x'|\le \delta^{j}$ for some $j\le j(a)$. Then, according to the previous discussion, we have $|x^n-y^n|\le \theta^j a=\delta^{j\gamma}a\le \delta^{-\gamma} |x'|^\gamma a$. When $|x'|<\delta^{j(a)}$ we just use $|x^n-y^n|\le \theta^{j(a)}a$. Altogether, we have
	\[|x^n-y^n|\le C_\delta a \min( \theta^{j(a)}, |x'|^\gamma), \qquad \text{for any $(x',x^n)\in \partial \widetilde E\cap B_1(0)$},\]
	with $\theta^{j(a)}\to 0$ as $k\to +\infty$ (or equivalently as $a\to 0$).
	
	Now, if $y$ is any point of $ B_{1/2}(0)$, we notice that we can apply the iteration procedure to the set $\widetilde E-y$ and obtain the same estimates. Indeed, for such points we have
	\[|y\cdot e_j|\le a 2^{j(1+\alpha)}, \qquad \text{for any $j\in \{0,\ldots,k\}$.}\]
	Hence, for any $x\in \partial \widetilde E\cap B_{2^{j-1}}(y)$, since $x\in \partial \widetilde E\cap B_{2^{j}}(0)$ we have 
	\[|(y-x)\cdot e_j|\le 2a 2^{j(1+\alpha)}, \qquad \text{for any $j\in \{0,\ldots,k\}$},\]
	thus we are in the setting of Proposition \ref{prop:harnack} (with slightly increased constants, but this is harmless), and we can iterate Harnack inequality starting from $\widetilde E-y$. \end{proof}

\begin{rem}We cannot write $|y^n-x^n|\le C\max(b_k^\gamma,|y'-x'|^\gamma)$ for any $(x,y)\in \partial \widetilde E_k^*\cap (B_{1/2}'(0)\times \mathbf R)$ because we have no information about the set $ \widetilde E_k$ outside $B_{2^k}(0)$. This causes no harm because we will prove convergence of the sequence on any compact subset of $B_{1/2}'(0)\times \mathbf R$, and any of these is contained in $B_{1/2}^*$ for $k$ large enough.\end{rem}
\subsection{Convergence to a limit function}

Using Proposition \ref{approximation partialE}, an application of Arzel{\`a}--Ascoli theorem allows to show compactness of the sequence $\{ \widetilde E_k^*\}$. The proof is done in \cite{CRS} but we provide a bit more details.
\begin{prop}\label{convergence}Up to a subsequence, $\{\widetilde E_{k}^*\}$ converges uniformly on compact subsets of $B_{1/2}'\times \mathbf{R}$ to the subgraph of a Hölder continuous function $f:B_{1/2}'(0)\to \mathbf R$. More precisely, there is a sequence of integers $k_j\to +\infty$ such that for any $\varepsilon>0$ and any compact subset $K\subset B_{1/2}'(0)\times \mathbf R$, if $j$ is large enough, we have
	\[\{x^n\le f(x')-\varepsilon\}\subset \widetilde E_{k_j}^*\subset \{x^n\le f(x')+\varepsilon\}~\text{in $K$}.\]\end{prop}

\begin{rem}An immediate consequence is that for any compact subset $K\subset B_{1/2}'\times \mathbf R$, 
	\[\sup_{(x',x^n)\in \partial \widetilde E_{k_j}^*\cap K} |x^n-f(x')|\underset{j\to +\infty}\longrightarrow 0,\]
 \emph{i.e.} the sequence of boundaries $(\widetilde \partial E_{k_j}^*)$ converges uniformly on compact subsets of $B_{1/2}'\times \mathbf R$ to the graph of $f$.
	This result is weaker than Proposition \ref{convergence}, and they are \emph{a priori} not equivalent since the sets $E_k$  may not enjoy uniform density estimates, as we discussed in Remark \ref{rem: uniform density }. \end{rem}

\begin{proof}[Proof of Proposition \ref{convergence}]
	Let $C$ be as in Proposition \ref{approximation partialE} and let \[f_k^+(x'):=\sup \big\{y^n-C|y'-x'|^\gamma, \ (y',y^n)\in \partial \widetilde E_k^*\cap B_{1/2}^*\big\}\] and \[f_k^-(x'):=\inf \big\{y^n+C|y'-x'|^\gamma, \ (y',y^n)\in \partial \widetilde E_k^*\cap B_{1/2}^*\big\}.\]
	By construction, we have
	\begin{equation}\label{byconstruction}\{x^n\le f_k^-(x')\}\subset \widetilde E_k^*\subset \{x^n\le f_k^+(x')\}~\text{in $B_{1/2}^*$}.\end{equation}
	
	\textit{Step 1. The functions $f_k^+$ and $f_k^-$ are $\gamma$-Hölder continuous}, with uniform Hölder estimates that are independent of $k$. It is clear because $f_k^+$ (resp. $f_k^-$) is defined as the supremum (resp. infimum) of a family of uniformly $\gamma$-Hölder functions.
	
	\textit{Step 2. Controlling $|f_k^+-f_k^-|$.} Let us show
	\begin{equation}\label{diff holder} |f_k^+-f_k^-|\le Cb_k^\gamma,\end{equation}
	where $C$ is above. We only prove that $f_k^+\le f_k^-+Cb_k^\gamma$ since the other inequality is handled similarly. For any $y=(y',y^n)$ and $z=(z',z^n)$ in $\partial \widetilde E_k^*\cap B_{1/2}^*$ we have, according to Proposition~\ref{approximation partialE},
	\[y^n-z^n\le C\max(b_k^\gamma,|y'-z'|^\gamma)\le Cb_k^{\gamma}+C|x'-y'|^\gamma+ C|x'-z'|^\gamma.\]
	We infer 
	\[y^n-C|x'-y'|^\gamma \le z^n+ C|x'-z'|^\gamma+Cb_k^{\gamma}.\]
	Taking infimum in $z$ and supremum in $y$ yields $f_k^+(x')\le  f_k^-(x')+Cb_k^{\gamma}.$ 
	
	\textit{Step 3. Letting $k\to +\infty$.} By Arzel{\`a}--Ascoli theorem, up to a subsequence, $(f_k^+)$ and $(f_k^-)$ converge in $L^{\infty}_{\mathrm{loc}}$ to $\gamma$-Hölder continuous functions $f^+$ and $f^-$. It follows from (\ref{diff holder}) that $f^+=f^-$. We conclude the proof by combining \eqref{byconstruction} with the fact that $B_{1/2}^*$ contains any given compact subset of $B_{1/2}'\times \mathbf{R}$ for sufficiently large $k$.
\end{proof}
Actually, the estimates above can be conducted in larger and larger balls. Indeed, \eqref{introset} implies that 
\begin{equation}\label{diff vectors}|\nu_j^k-\nu_{j+1}^{k}|\le C2^{-j\alpha},\end{equation}
for some constant $C$ independent of $j$ and $k$. Fix $l\in \mathbf N$. Recalling $\nu_k^k=e_n$, we infer from \eqref{diff vectors} that $|e_n-\nu_{k-l}^k|\le C 2^{-(k-l)\alpha}$, implying that in $B_{2^{-(k-l)}}$ we have
\[ \{x^n\le -C2^{-(k-l)(\alpha+1)}\}\subset E_k\subset \{x^n\le C2^{-(k-l)(\alpha+1)}\},\]
where $C$ depends only on $n$. Proposition \ref{convergence} then shows that, up to a subsequence, the sequence of sets
\[2^{k-l}\Big\{\big(x',\frac{x^n}{a_{k-l}}\big) \ \big| \ (x',x^n)\in E_k\Big\}\]
converges, uniformly on compact subsets of $B_{1/2}'\times \mathbf R$, to the subgraph of a Hölder continuous function. After proper rescaling this shows that, up to a subsequence, $\{ \widetilde E_k^*\}$ converges, on any compact subset of $B_{2^{l-1}}'\times \mathbf R$, to the subgraph of a Hölder continuous function. Since this holds for any $l$, we obtain the following proposition.

\begin{prop}\label{convergencecompact} Up to a subsequence, $\{\widetilde E_k^*\}$ converges, uniformly on compact subsets of $\mathbf R^n$, to the subgraph of a continuous function $f:\mathbf{R}^{n-1}\to \mathbf{R}$. Moreover, $f(0)=0$ and we have a growth estimate
	\[|f(x')|\le C(1+|x'|^{1+\alpha}).\]\end{prop}

\begin{proof} The first point is a standard diagonal argument. The fact that $f(0)=0$ follows from the assumption  $0\in \partial E_k$. To prove the growth estimate, we work with the rescaled sets $\partial \widetilde E_k$.
	Recall that, letting $e_l^k:=\nu_{k-l}^k$ we have
	\[\partial \widetilde E_k\cap B_{2^l}(0)\subset \{|x\cdot e_l^k|\le a_k2^{l(1+\alpha)}\}, \qquad \text{for any $l\in \{0,\ldots,k\}$},\]
	with $e_0^k=e_n$ and $|e_l^k-e_n|\le C a_k2^{l\alpha}$. Let $p_l^k:=\frac{1}{a_k}(e_l^k-(e_l^k\cdot e_n) e_n)$. Then $|p_l^k|\le C2^{l\alpha}$, and one can show that
	\[\partial \widetilde E_k\cap B_{2^l}(0)\subset \left\{(x',x^n), \ \left|\frac{x^n}{a_k}+p_l^k\cdot x'\right|\le C2^{l(1+\alpha)}\right\}.\]
	Since $(p_l^k)_k$ is bounded, as $a_k\to 0$, up to extracting a subsequence, $(p_{l}^{k})_k$ converges to $p_l\in \mathbf{R}^{n-1}$ and we find that
	\[|f(x')+p_l\cdot x'|\le C2^{l(1+\alpha)}\]
	in $B_{2^l}'(0)$. We conclude that for any $l\ge 0$, we have 
	\[|f(x')|\le C2^{l\alpha}(|x'|+2^l) ~\text{in $B_{2^l}'(0)$},\]
	hence the growth estimate on $f$.
\end{proof}

\subsection{The limit function is linear}


Let us recall the definition of viscosity solutions to $\Delta^su=0$ (here $\Delta$ denotes the standard Euclidean Laplacian).

\begin{defi}Let $s\in (0,1)$, and let $u$ be a measurable function $u:\mathbf R^n\to \mathbf R^n$ satisfying the integrability condition
	\[\int \frac{|u(x)|}{(1+|x|^2)^{\frac{n+s}{2}}}\mathrm dx<+\infty.\]
	Then, for any $x\in \mathbf R^n$, the fractional Laplacian  $\Delta^su(x)$ (defined in \eqref{eq: fracLap pv}) is well-defined as a principal value as long as $u$ is touched by above or below by a smooth function at $x$.
	
	We say that $\Delta^s u=0$ in the viscosity sense if $\Delta^s u(x)\le 0$ whenever $u$ is touched by below by a smooth function at $x$, and $\Delta^s u(x)\ge 0$ whenever $u$ is touched by above by a smooth function at $x$. \end{defi}

We shall use the following result:
\begin{prop}[{\cite[Proposition 6.7]{CRS}}] Assume $u:\mathbf R^{n-1}\to \mathbf R$ satisfies the growth condition $|u(x)|\le 1+|x|^{1+\alpha}$ for some $\alpha<s$, and $\Delta^{\frac{1+s}{2}}u=0$ in the viscosity sense in $\mathbf R^{n-1}$. Then $u$ is linear. \end{prop}

\begin{rem}By performing a linear change of variables, we see that the result still holds when replacing the Euclidean Laplacian by the Laplacian associated to any constant metric.\end{rem}

 Denote by $\iota$ the embedding $\mathbf R^{n-1}\to \mathbf R^n, x'\mapsto (x',0)$. If $g$ is a metric on $\mathbf R^n$, it induces by pullback a metric $\iota^*g$ on $\mathbf R^{n-1}$.
	
	\begin{prop}\label{limlinear}There is a constant metric $h$ on $\mathbf R^n$ such that the limit function $f$ satisfies
		\[\Delta_{\iota^*h}^{\frac{1+s}{2}}f=0\]
		in the viscosity sense; hence $f$ is linear. \end{prop}

\begin{proof}We refer again to \cite{CRS} for the proof in the Euclidean setup. Fix $\varepsilon>0$. Let $\varphi$ be a smooth function that touches $f$ by below at a point $\tilde y_0=2^ky_0$. It means that $f(\tilde y_0')=\varphi(\tilde y_0')$ and
	\[0\le f(x')-\varphi(x') \ \ \ \ \text{for any $x'\in \mathbf{R}^n$.}\]
	Fix $R=2^l>0$ large enough and $\delta>0$ small. According to Proposition \ref{convergencecompact}, one can find $E=E_k$, and $a=a_k$ very small (we take $a=a(\varepsilon)$ such that $a\to 0$ when $\varepsilon\to 0$), such that $\widetilde E=2^k E$ is included in a $a\varepsilon$-neighborhood of the subgraph of $a f$ (and vice versa) in $D_R(\tilde y_0):=\{|x'-\tilde y_0'|\le R\}\times \{|x^n-\tilde y_0^n|\le R\}$. It means that
	\begin{equation}\label{f proche de E} \{x^n\le af(x')-a\varepsilon\}\subset \widetilde E\subset \{x^n\le af(x')+a\varepsilon\} \ \text{in $D_R(\tilde y_0)$}.\end{equation}
	We take $a,\varepsilon$ small enough to ensure $a,\varepsilon\ll \delta$.
	The strategy is now to find a graph that touches $\partial E$ near $y_0$; since $E$ has bounded NMC we will deduce that $f$ is a viscosity solution to a nonlocal equation in one dimension less.  To achieve this, we introduce $\psi(x')=\varphi(x')-|x'-\tilde y_0'|^2$. It pushes points that are far from $\tilde y_0$ away from the graph of $af$, hence from $\partial \widetilde E$. Then, $\partial \widetilde E$ is touched by below by a vertical translation of $a\psi$ at a point $\tilde y_1$ (depending on $k$) and one can show
	\begin{equation} \label{dist y1y2}|\tilde y_1'-\tilde y_0'|^2\le \varepsilon.\end{equation}
	Notice that we can shift the estimate (\ref{f proche de E}) to $\tilde y_1$ to obtain, up to replacing $R$ by $R-\varepsilon$, 
	\begin{equation}\label{f proche de E'}\partial \widetilde E\cap D_R(\tilde y_1)\subset \{\tilde y_1+z : \ |z^n-a(f(\tilde y_1'+z')-f(\tilde y_1'))|\le 2a\varepsilon\}.\end{equation}
	Thanks to Proposition \ref{approxeulerlagrange}, the NMC boundedness assumption at the point $ y_1=2^{-k}\tilde y_1$ gives
	\begin{equation}\label{NMCy1}\mathrm{p.v.}\int \big(\chi_E(x)-\chi_{\mathcal CE}(x)\big)K_{g_k(y_1)}(x, y_1) \mathrm dx\le C,\end{equation}
	with $C$ depending only on $C_0,n,s$. We will split \eqref{NMCy1} in three terms. The first one consists of points far from $y_1$, whose contribution is controlled by tail estimates. The second one consists of points located in a small neighborhood of $y_1$. We obtain a bound by below because $\partial E$ is touched by below by a smooth graph at the point $y_1$. The last term consists of points $x$ living at the intermediate scale $d(\tilde x,\tilde y_1)\in ]\delta,R[$. In this region we will use that $\partial \widetilde E$ is close to the graph of $af$. 
	
	Again, we let, for $r>0$, 
	\[D_r(y_1):=\{|x'- y_1'|\le r\}\times \{|x^n-y_1^n|\le r\}.\]
	First we estimate
	\[I_1:=\mathrm{p.v.}\int_{D_{\delta2^{-k}}(y_1)} (\chi_E-\chi_{\mathcal C E})(x)K_{g_k(y_1)}(x, y_1)\mathrm dx.\]
	In $D_{\delta}(\tilde y_1)$, we use the fact that $\partial \widetilde E$ is touched by below by a vertical translation $\widetilde P$ of the subgraph of $a\psi$. Since $\chi_{\widetilde E}-\chi_{\mathcal C\widetilde E}\ge \chi_{\widetilde P}-\chi_{\mathcal C\widetilde P}$ in $D_\delta(\tilde y_1)$, as in (\ref{I3}) we get,
	\[I_1\ge 2^{ks}\mathrm{p.v.}\int_{D_{\delta}(\tilde y_1)} (\chi_{\widetilde P}-\chi_{\mathcal C \widetilde P})( x)K_{g_k(y_1)}( x, \tilde y_1)\mathrm d x\ge -2^{ks} C_{n,s,\psi}a \delta^{1-s}.\]
	Let $I_2$ denote the contribution of points in $\mathcal C D_{R2^{-k}}(y_1)$ to \eqref{NMCy1}, that is
	\[I_2:=\int_{\mathcal C D_{R2^{-k}}(y_1)} (\chi_E-\chi_{\mathcal C E})(x)K_{g_k(y_1)}(x, y_1)\mathrm dx.\]
	Combining $(\ref{tail})$ and $(\ref{dyadic})$, we have for $k\ge l$ (recall that $R=2^l$),
	\[|I_2|\le C_{n,s}+C_{n,s}2^{(k-l)(s-\alpha)}\le C_{n,s}2^{k(s-\alpha)}R^{\alpha-s}.\]
	It remains to compute
	\[I_3=\int_{D_{R2^{-k}}(y_1)\setminus D_{\delta2^{-k}}(y_1)} (\chi_E-\chi_{\mathcal C E})(x)K_{g_k(y_1)}(x, y_1)\mathrm dx.\]
	We recall that $K_{g_k(y_1)}(x,y_1)=\alpha_{n,s}|x-y_1|_{g_k(y_1)}^{-(n+s)}$. One can check that if $x\in \partial \widetilde E\cap \big( D_{R}(\tilde y_1)\setminus D_{\delta }(\tilde y_1)\big)$ then
	\begin{equation}\label{borne2}\big| |x- \tilde y_1|_{g_k(y_1)}^{-(n+s)}-|x'-\tilde y_1'|_{g_k(y_1)}^{-(n+s)} \big|\le C_{R,\delta} a^2.\end{equation}
	Indeed, if $x\in \widetilde E\cap \big( D_{R}(\tilde y_1)\setminus D_{\delta }(\tilde y_1)\big)$ then $|x'|\ge \delta$ and $|x^n|\le C_Ra$, as follows from the fact that $\partial \widetilde E$ is close to the graph of $af$ in $D_{R}$. This implies \eqref{borne2}. Using the approximation (\ref{f proche de E'}) together with (\ref{borne2}) we obtain (recall $a=2^{-k\alpha}$)
	\[I_3= \alpha_{n,s}2^{k(s-\alpha)}\left(\displaystyle\int_{B_R'(0)\setminus B_{\delta}'(0)} \frac{2(f(z'+ \tilde y_1')-f(\tilde y_1'))}{|z'|_{g_k(y_1)}^{n+s}}\ \mathrm dz'+O(\varepsilon)+O(a)\right).\]
	Recall from (\ref{NMCy1}) that $I_1+I_2+I_3\le C$. Multiply both sides by $2^{k(\alpha-s)}$, then let $\varepsilon$ go to $0$ and $k$ go to $+\infty$ (implying $a=2^{-k\alpha}\to 0$) along an appropriate sequence $(k_j)_{j\ge 0}$ to obtain 
	\[\limsup_{k_j\to +\infty} \int_{B_R'(0)\setminus B_{\delta}'(0)}\frac{(f(z'+\tilde y_0')-f(\tilde y_0'))}{|z'|_{g_{k_j}(y_1)}^{n+s}}\ \mathrm dz'\le C_{n,s}(R^{\alpha-s}+\delta^{1-s}).\]
	Using (\ref{dist y1y2}) we have $\tilde y_1\to \tilde y_0$ as $k\to +\infty$. Also, $|y_1|\le 2^{-k+1}$. Up to extracting a subsequence, we can assume that $(g_{k_j}(0))$ converges to some positive definite matrix $h$. Then, since $\|\mathrm Dg_{k_j}\|_{L^\infty}\le 1$, we have
	\[\|g_{k_j}(y_1)-h\|\le \|g_{k_j}(0)-h\|+\|g_{k_j}(0)-g_{k_j}(y_1)\|\le  \|g_{k_j}(0)-h\|+2^{-k_j+1}.\]
	Therefore $\big(g_{k_j}(y_1)\big)$ also converges to $h$ as $k_j\to +\infty$. Since $f$ is continuous, we can pass to the limit inside the integral, giving
	\[\int_{B_R'\setminus B_{\delta}'}\frac{(f(z'+\tilde y_0')-f(\tilde y_0'))}{|z'|_h^{n+s}}\ \mathrm dz'\le C_{n,s}(R^{\alpha-s}+\delta^{1-s}).\] 
	We obtain the desired result by letting $\delta\to 0, R\to +\infty$. We stress out that the metric $h$ does not depend on the point $\tilde y_0$.
\end{proof}

\subsection{Proof of Theorems \ref{improvement} and \ref{mainthm}} 
\begin{proof}[Proof of Theorem \ref{improvement}] 
	By contradiction, assume that we have a sequence $\{E_k\}$ of sets with uniformly bounded nonlocal mean curvature\footnote{We recall that the metric may depend on $k$, although it always satisfies the assumptions of Theorem \ref{improvement}.} in $B_1$, with $a_k=2^{-k\alpha}\to 0$, such that for all $\nu\in \mathbf S^{n-1}$ one of the inclusions
	\[\{x\cdot \nu \le -a_k2^{-(\alpha+1)}\}\cap B_{1/2}\subset \widetilde E_k\cap B_{1/2} \subset \{x\cdot \nu \le a_k2^{-(\alpha+1)}\}\]
	fails. This contradicts the fact that, up to a subsequence, $\{\widetilde E_k^*\}$ converges uniformly in $B_{1/2}(0)$ to the subgraph of a linear function.  
	
\end{proof}
Now we are ready to prove our main theorem. This is a standard application of Theorem \ref{improvement}, but we include of proof for the sake of completeness.
\begin{proof}[Proof of Theorem \ref{mainthm}] We may assume $r=1$, since the general result follows by rescaling. 
	\item \textit{Step 1. } Take $\sigma= \frac 12 2^{-k(\alpha+1)}$ with $k\ge k_0$, where $k_0$ is as in Theorem \ref{improvement}. Let $g$ be a smooth Riemannian metric on $\mathbf R^n$ satisfying the assumptions of Theorem \ref{mainthm} for $r=1$. Also, let $E\subset \mathbf R^n$ be a measurable subset with NMC bounded by $C_0$ in the viscosity sense in $B_1(0)$, for the metric $g$. Moreover, assume $0\in \partial E$ and
	\[\{x^n\le -\sigma\}\cap B_1(0)\subset E\cap B_1(0)\subset \{x^n\le \sigma\}.\]
	Then for any $x\in \partial E\cap \big(B_{1/2}'(0)\times [-\sigma,\sigma]\big)$ we have
	\[\{(z-x)\cdot e_n\le -2\sigma\}\subset  E\subset \{(z-x)\cdot e_n\le  2\sigma\} \ \text{in $B_{1/4}(x)$}.\]
	In turn, for any $j\in \{2,\ldots,k\}$, 
	\[\{(z-x)\cdot e_n\le -2^{-j(1+\alpha)}\}\subset  E\subset \{(z-x)\cdot e_n\le  2^{-j(1+\alpha)}\} \ \text{in $B_{2^{-j}}(x)$}.\]
	Therefore, if $k\ge k_0$ is large enough, then for any $x\in \partial E\cap \big(B_{1/2}'(0)\times [-\sigma,\sigma]\big)$ we can apply Theorem \ref{improvement} to get a sequence of unit vectors $(\nu_j(x))_{j\ge 2}$, with $\nu_j(x)=e_n$ for $2\le j\le k$, such that for all $j\ge 2$,
	\[\{(z-x)\cdot \nu_j(x)\le -2^{-j(\alpha+1)}\}\subset  E\subset \{(z-x)\cdot \nu_j(x)\le  2^{-j(\alpha+1)}\} \ \text{in $B_{2^{-j}}(x)$}.\]
	The inclusions above imply that $(\nu_j(x))$ converges with a geometric rate to a unit vector $\nu(x)$. More precisely, we have $|\nu_j(x)-\nu(x)|\le C 2^{-j\alpha}$ for some universal constant $C$. It follows that
	\[\{(z-x)\cdot \nu(x)\le -C 2^{-j(\alpha+1)}\}\subset  E\subset \{(z-x)\cdot \nu(x)\le C 2^{-j(\alpha+1)}\}~\text{in $B_{2^{-j}}(x)$}.\]
	Then, up to increasing slightly $C$, for any $\rho\in (0,1/4)$,
	\begin{equation}\label{inclusion Brho}\{(z-x)\cdot \nu(x)\le -C \rho^{1+\alpha}\}\subset E\subset \{(z-x)\cdot \nu(x)\le C \rho^{1+\alpha}\}~\text{in $B_\rho(x)$.}\end{equation}
	We infer that for all $x,y$ in $\partial E\cap \big(B_{1/2}'(0)\times [-\sigma,\sigma]\big)$,
	\begin{equation}\label{nu}|\nu(x)-\nu(y)|\le C|x-y|^{\alpha}.\end{equation} 
	If $k$ is large enough then $|e_n-\nu(x)|\le \frac 12$ for all $x\in \partial E\cap \big(B_{1/2}'(0)\times [-\sigma,\sigma]\big)$. It implies with \eqref{inclusion Brho} that $\partial E\cap \big(B_{1/2}'(0)\times [-\sigma,\sigma]\big)$ is the graph of a function $f:B_{1/2}'\to [-\sigma,\sigma]$. We point out that since $\partial E$ is closed, $\partial E\cap B_{1}(0)$ is compact, thus $f$ is automatically continuous. Actually, the above inclusions imply that $f$ is differentiable. Letting $x=(x',f(x'))$ and writing $\nu=(\nu',\nu_n)\in \mathbf R^{n-1}\times \mathbf R$ we have
	\begin{equation}\label{defnu} \nabla_{x'} f= \frac{-\nu'(x)}{\sqrt{1-|\nu'(x)|^2}}.\end{equation}
	Since $|\nu'|\le \frac 12$ in $\partial E\cap \big(B_{1/2}'(0)\times [-\sigma,\sigma]\big)$, we have the gradient upper bound $|\nabla_{x'}f|\le 1$, thus for any $x,y\in \partial E\cap \big(B_{1/2}'(0)\times [-\sigma,\sigma]\big)$, we have 
	\[|x-y|\le |x'-y'|+|f(x')-f(y')|\le \big(1+\|\nabla f\|_{L^\infty(B_{1/2}'(0))}\big)|x'-y'|.\] 
	Eventually, we deduce from \eqref{nu},\eqref{defnu} that
	\[|\nabla_{x'} f-\nabla_{y'} f|\le C |x'-y'|^{\alpha}\]
	in $B_{1/2}'(0)$, where $C$ depends only on $k_0,n,\alpha,s$.
\end{proof}

\section{From $\mathbf R^n$ to arbitrary manifolds}\label{fromto} In this section, we explain how Theorem \ref{maincor} for hypersurfaces of bounded NMC in arbitrary manifolds follows from Theorem \ref{mainthm}. The idea is that even though the NMC is defined by an integral over the whole manifold -- thus is nonlocal in nature --, boundedness of NMC is a very local property, both regarding the set under consideration, and the ambient manifold. 

From now on, we fix an arbitrary smooth, connected, orientable, Riemannian manifold $(M,g)$. In the whole section, we consider fixed constants $0<\eta_0<\rho_0<1<R_0$ such that $2-\rho_0<R_0$. They will be chosen explicitly in the proof of Theorem \ref{maincor}. Given some radius $r$ that may vary, we shall denote $\eta=r\eta_0, \rho=r\rho_0, R=rR_0$. The main result of this section is the following.

\begin{prop} \label{regarderlocalement}  Assume $\mathrm{FA}_1(M,g,p,R,\varphi)$. Consider a measurable subset $E\subset M$. Assume that $E$ has  NMC bounded by $C_0r^{-s}$ in $V_{\eta}(p)=\varphi(B_{\eta}(p))$, in the viscosity sense. Let $F\subset \mathbf R^n$ be defined by $F=\varphi^{-1}(E\cap V_R(p))$. Then there exists a smooth Riemannian metric $h=(h_{ij})$ on $\mathbf R^n$ such that
	\begin{itemize}
		\item $h=\varphi^*g$ in $B_r(0)$.
		\item $\frac 12\le h\le 2$ and $r\|\nabla h_{ij}\|_{L^{\infty}(\mathbf R^n)}\le C$ for some $C=C_{R_0}$. 
		\item $F$ has NMC bounded by $(C_0+C_{n,s})r^{-s}$ in the viscosity sense in $B_\eta(0)$, for the metric $h$.
	\end{itemize}
\end{prop}

In the following, we do not write dependencies of constants with respect to $\eta_0,\rho_0,R_0$, since these will be fixed explicitly.
\begin{proof}[Proposition \ref{regarderlocalement} implies Theorem \ref{maincor}]
	This is a straightforward application of Theorem \ref{mainthm} combined with Proposition \ref{regarderlocalement}. One can choose universal constants $\eta_0=3/4$, $\rho_0=7/8$, $R_0=5/4$ in the above statements (notice that $R_0+\rho_0>2$).
\end{proof}

In the following, if $\Omega$ is an open subset of $(M,g)$, we denote by $H_{\Omega}$ (or $H_{\Omega,g}$) the Dirichlet heat kernel in $\Omega$ (\emph{i.e.} with zero boundary conditions). The parabolic maximum principle implies that $H_{\Omega}(t,p,q)\le H(t,p,q)$ for any $p,q\in \Omega$.

Let us describe the strategy to prove Proposition \ref{regarderlocalement}: first, we provide quantitative estimates on the mass of the heat kernel $H_g(t,p,\cdot)$ outside $V_r(p)$, under some flatness assumption around $p$ (Proposition \ref{intégrale dehors}). Then, we show that when $q\in V_\rho(p)$, the heat kernel $H_g(t,p,q)$ is well approximated at small times by $H_{V_r(p),g}(t,p,q)$ (Proposition \ref{corimportant}). This allows considering kernels depending only on the geometry of $M$ in $V_r(p)$. Next, we map $E\cap V_R(p)$ by $\varphi^{-1}$ to a subset $F\subset B_R(0)\subset \mathbf R^n$. We construct a metric $h$ satisfying the first two items of Proposition \ref{regarderlocalement} by extending smoothly $\varphi^*g$ outside $B_r(0)$. In the coordinates given by $\varphi$, the Dirichlet heat kernel $H_{V_r(p),g}$ is simply given by the Dirichlet heat kernel $H_{B_r(0),h}$. Eventually, by applying Propositions \ref{intégrale dehors} and \ref{corimportant}, we show that $F$ has bounded NMC in the viscosity sense for the metric $h$ in $B_{\eta}(0)$.

\subsection{$L^1$ bounds on the heat kernel under a local flatness assumption}  Although we could try to provide pointwise upper bounds on the heat kernel, since we are dealing with integral equations we only need some $L^1$ bounds.

\begin{prop}[Tail estimates]\label{intégrale dehors}Let $p\in M$. Assume $\mathrm{FA}_1(M,g,p,r,\varphi)$. Then, for any $q\in V_\eta(p)=\varphi(B_{\eta}(0))$ and $t\le r^2$ we have
	\[\int_{M\setminus V_\rho(p)} H(t,q,z)\mathrm dV(z)\le C \mathrm{e}^{-c \frac{r^2}{t}}\]
	and
	\[\label{eqintégraledehors}\int_{M\setminus V_{ \rho}(p)} K(q,z) \mathrm dV(z)\le C r^{-s}.\]
	The constants depend only on $n,s$.\end{prop}

\begin{rem}Of course, the constants depend on $\eta_0$ and $\rho_0$, but these will be explicitly fixed. \end{rem}
\begin{rem}We stress out that we make no assumption at all on the geometry of the manifold outside $V_r(p)=\varphi(B_r(0))$. \end{rem}
\begin{proof} Let $\Delta_{V_r(p)}$ denote the Dirichlet Laplacian in $V_r(p)$, and let $H_{V_r(p)}$ denote the associated heat kernel.
	
	\textit{1.} We claim that for any $(q,z)\in V_r(p)$ we have
	\begin{equation}\label{heatcompact}H_{V_r(p),g}(t,q,z)\le C t^{-n/2}\mathrm{e}^{-\frac{c d(q,z)^2}{t}}. \end{equation}
	for some small universal constant $c>0$ and $C=C_{c,n}$. Indeed, since $M$ satisfies $\mathrm{FA}_1(M,g,p,r,\varphi)$, we can map $V_r(p)$ quasi-isometrically by $\varphi^{-1}$ to $B_r(0)\subset \mathbf R^n$, then extend smoothly $\varphi^*g$ to a metric $h$ on $\mathbf R^n$ satisfying $\frac 12\le h\le 2$. Then, for any $x,y\in B_r(0)$ we have \[H_{V_r(p),g}(t,\varphi(x),\varphi(y))=H_{B_r(0),h}(t,x,y).\]
	Now we have $H_{B_r(0),h}\le H_h$, and we can use Proposition \ref{FBmetric} to obtain Gaussian upper bounds on $H_h$. Combined with the fact that $\varphi$ is a quasi-isometry $B_r(0)\to V_r(p)$, this gives \eqref{heatcompact}.
	
	\medskip
	
	\textit{2.} Take $\eta<\delta<\rho$, typically $\delta=\frac{\eta+\rho}{2}$, and a cutoff function $\chi$ such that $\chi\equiv 1$ in $V_\delta(p)$ and $\chi\equiv 0$ outside $V_\rho(p)$. We can ask for the second derivative bound $|\Delta_g\chi|\le C r^{-2}$ with $C$ depending only on $n,\eta_0,\rho_0$. Let 
	\begin{equation}\label{defu}u(t,q):= \mathrm{e}^{t\Delta_{V_r(p)}} \chi(q)=\int_{V_r(p)} H_{V_r(p)}(t,q,z)\chi(z) \mathrm dV(z).\end{equation} Then, differentiating with respect to the time variable,
	\[\partial_t u(t,q)=\int_{V_r(p)} \chi(z)\partial_t H_{V_r(p)}(t,q,z)\mathrm dV(z).\]
	Now, since $\partial_t H_{V_r(p)}(t,q,z)=\Delta_g H_{V_r(p)}(t,q,z)$ and $\chi$ vanishes on $V_r(p)\setminus V_{\rho}(p)$, we obtain by Green's formula,
	\[\partial_t u(t,q)=\int_{V_r(p)} \Delta_g\chi(z)H_{V_r(p)}(t,q,z)\mathrm dV(z).\]
	Integrating over $\tau\in[0,t]$, noticing meanwhile that $\Delta_g \chi$ is supported in $\overline{ V_\rho(p)} \setminus V_{\delta}(p)$, we get that for any $q\in V_r(p)$,
	\begin{equation}\label{eq: intégration temps}u(t,q)-u(0,q)=\int_0^t \int_{V_\rho(p)\setminus V_\delta(p)} \Delta_g \chi(z)H_{V_r(p)}(\tau,q,z)\mathrm dV(z)\mathrm d\tau.\end{equation}
	Assume $q\in V_\delta(p)$. On the one hand $u(0,q)=\chi(q)=1$, and it follows from \eqref{eq: intégration temps} that
	\begin{equation}\label{ineq1} u(t,q)\ge 1-\|\Delta_g \chi\|_{L^{\infty}(V_\rho(p)\setminus V_{\delta}(p))}\int_0^t \int_{V_\rho(p)\setminus V_{\delta}(p)}H_{V_r(p)}(\tau,q,z)\mathrm dV(z)\mathrm d\tau.\end{equation}
	On the other hand $\chi$ is supported in $\overline{V_\rho(p)}$, and $H_{V_r(p)}\le H$. Since $H$ has total mass $\le 1$, using expression \eqref{defu} we have
	\begin{equation}\label{ineq2} u(t,q)\le 1-\int_{M\setminus V_\rho(p)} H(t,q,z)\mathrm dV(z).\end{equation}
	Thus, for $q\in V_{\delta}(p)$, combining (\ref{ineq1}) and (\ref{ineq2}),
	\[\int_{M\setminus V_\rho(p)} H(t,q,z)\mathrm dV(z)\le \|\Delta_g \chi\|_{L^{\infty}(V_\rho(p)\setminus V_{\delta}(p))}\int_0^t \int_{V_{\rho}(p)\setminus V_{\delta}(p)}H_{V_r(p)}(\tau,q,z) \mathrm dV(z)\mathrm d\tau. \]
	If additionally $q\in V_{\eta}(p)$, then for any $z\in V_\rho(p)\backslash V_\delta(p)$ we have 
	\[d(q,z)\ge \frac 12 (\rho-\eta)\ge \frac{\rho_0-\eta_0}{2}r.\]
	By using \eqref{heatcompact}, the bound $|V_\rho(p)|\le Cr^n$ (recall that $\varphi$ is a quasi-isometry) and the bound on $|\Delta_g \chi|$ we infer
	\begin{equation}\label{intinBr} \int_{M\setminus V_{\rho}(p)} H(t,q,z)\mathrm dV(z)\le C_n r^{n}\int_0^t \tau^{-n/2}\mathrm{e}^{-cr^2/\tau}\mathrm d\frac{\tau}{r^2}\le C_n \int_0^{t/r^2} u^{-\frac n2}  \mathrm{e}^{-c/u}\mathrm du,\end{equation}
	where $c$ has been taken smaller than in \eqref{heatcompact}. If $t\le r^2$ then the integral in the last line is smaller than $C_n\mathrm{e}^{-cr^2/t}$ for some smaller $c>0$, implying the result. Actually, since the heat kernel has mass $\le 1$, the integral bound holds for any $t$.

	\textit{3.} To show the integral bound on $K(q,\cdot)$, just use \eqref{intinBr} to get
	\[ \int_{\mathbf R_+} \frac{\mathrm dt}{t^{1+s/2}}\displaystyle\int_{M\setminus V_{\rho}(p)}H(t,q,z) \mathrm dV(z) \le C_n \displaystyle\int_{0}^{+\infty} \mathrm{e}^{-cr^2/t}\frac{\mathrm dt}{t^{1+s/2}}\le C_{n,s} r^{-s}.\]
\end{proof}
\begin{rem} To define the cutoff function $\chi$, we start from a cutoff function $\tilde \chi$ on $\mathbf R^n$ such that $\tilde \chi\equiv 1$ in $B_\delta(0)$ and $\tilde \chi\equiv 0$ outside $B_\rho(0)$, satisfying $\|\mathrm D^2\tilde \chi\|_{L^\infty(\mathbf R^n)}\le Cr^{-2}$. We set $\chi=\varphi_*\tilde \chi$ in $V_r(p)$ and extend it by $0$ on $M$. Then, $\Delta_g \chi=\varphi_*(\Delta_{\varphi*g} \tilde \chi)$, implying $\|\Delta_g \chi\|_{L^\infty(M)}=\|\Delta_{\varphi^*g} \tilde \chi\|_{L^\infty(\mathbf R^n)}$, and by the flatness assumption we have \[|\Delta_{\varphi^*g} \tilde \chi|\le C_n r^{-1}\|\mathrm D \tilde \chi\|+ C_n\|\mathrm D^2\tilde \chi\|\le C_nr^{-2}.\]
\end{rem}
With Proposition \ref{intégrale dehors} at hand, we can focus on the contribution to $\mathrm H_s[E](p)$ of points living in $V_\rho(p)$. Under a flatness assumption at scale $r$ around $p$, up an error $\le C r^{-s}$ this contribution does not depend on the geometry of $M$ outside $V_r(p)$, see Proposition \ref{corimportant} below.

The following lemma relies on the parabolic maximum principle. It allows replacing the heat kernel $H(t,q,z)$ by the Dirichlet heat kernel $H_{V_r(p)}(t,q,z)$, that depends only on the geometry of $M$ in $V_r(p)$.
\begin{lem}\label{parabolic} Let $p\in M$ and let $\varphi:B_R(0)\to V_R(p)$ be a smooth diffeomorphism. Then,
	\[\sup_{q\in V_r(p)} \int_{V_\rho(p)}\big(H(t,q,z)-H_{V_r(p)}(t,q,z)\big) \mathrm dV(z)\le \sup_{(q,\tau)\in \partial V_r(p)\times [0,t]} \int_{M\setminus V_{r-\rho}(q)} H(\tau,q,z) \mathrm dV(z).\]
\end{lem}
\begin{proof}Let $u_0$ denote the indicator function of $V_{\rho}(p)$ and consider
	\[u(t,q):=(\mathrm{e}^{t\Delta}u_0-\mathrm{e}^{t\Delta_{V_r(p)}}u_0)(q)=\int_{M} (H(t,q,z)-H_{V_r(p)}(t,q,z))u_0(z) \mathrm dV(z).\]
	We have $(\partial_t-\Delta_g)u=0$, and $u(0,\cdot)=0$ in $V_r(p)$. Since $u$ is nonnegative everywhere (because $H_{V_r(p)}\le H$), we deduce from the parabolic maximum principle that for any $q\in V_r(p)$ and $t\ge 0$,
	\begin{equation}\label{apply pmp}u(t,q)\le \sup_{(w,\tau)\in \partial V_r(p)\times [0,t]} u(\tau,w)=\sup_{(w,\tau)\in \partial V_r(p)\times [0,t]} \mathrm{e}^{\tau\Delta}u_0(w),\end{equation}
	where we have used in the second equality the fact that $\mathrm{e}^{t\Delta_{V_r(p)}}u_0$ vanishes on $\partial V_r(p)$. Now, for any $w\in \partial V_r(p)$, we have $V_{r-\rho}(w)\subset M\backslash V_\rho(p)$, whence
	\[\mathrm{e}^{\tau \Delta_g} u_0(w)=\int_{V_{\rho}(p)} H(\tau,w,z) \mathrm dV(z)\le \int_{M\setminus V_{r-\rho}(w)} H(\tau,w,z) \mathrm dV(z).\]
	Combined with \eqref{apply pmp} this gives the result. We point out that we need $R>2r-\rho$ in order to make sense of $V_{r-\rho}(q)$ when $q\in \partial V_r(p)$.\end{proof}

\begin{rem}We can give a probabilistic interpretation of the result. Consider a Brownian motion $X_t$ starting at $q\in V_r(p)$, and let $\tau:=\inf \{t>0, \  X_t \in \partial V_r\}$. Then for $t>0$ we have
	\[\{X_t\in V_\rho\}=\{\tau>t \ \text{and } X_t\in V_\rho\}\sqcup \{\tau\le t\ \text{and } X_{t-\tau}\in V_\rho\}.\]
	It means that if a Brownian particle starting from $q\in V_r$ is in $V_\rho$ at time $t$, then either it has remained in $V_r$ during the whole trajectory, either the particle has hit $\partial V_r$ at some instant $\tau\le t$ and came back to $V_\rho$ in time $t-\tau$. This  translates to \eqref{apply pmp}.
	
\end{rem}
Thus, we are left to estimating the integrals $\int_{M\setminus V_{r-\rho}(z)} H(\tau,w,z) \mathrm dV(z)$, which is made possible by Proposition \ref{intégrale dehors}.

\begin{prop}\label{corimportant} Assume $\mathrm{FA}_1(M,g,p,R,\varphi)$. Then, for any $q\in V_{r}(p)$ and $t>0$,
	\[ \int_{V_\rho(p)}\big(H(t,q,z)-H_{V_r(p)}(t,q,z)\big) \mathrm dV(z)\le C \mathrm{e}^{-c \frac{r^2}{t}}.\]
	Moreover, denoting $K_{V_{r}(p)}(q,z):=\int \frac{\mathrm dt}{t^{1+s/2}} H_{V_r(p)}(t,q,z)\mathrm dt$, we have
	\[\int_{V_{\rho}(p)} |K(q,z)-K_{V_r(p)}(q,z)| \ \mathrm dV(z)\le Cr^{-s}.\]
	The constants depend only on $n$ and $s$. \end{prop}
\begin{proof}
	Recall that by Lemma \ref{parabolic},
	\begin{equation}\label{parabolic2}\sup_{q\in V_{r}(p)} \int_{V_{\rho}(p)}(H(t,q,z)-H_{V_r(p)}(t,q,z)) \mathrm dV(z)\le \sup_{(q,\tau)\in \partial V_r(p)\times [0,t]} \int_{M\setminus V_{r-\rho}(q)} H(\tau,q,z) \mathrm dV(z).\end{equation}
	Since $\mathrm{FA}_1(M,g,p,R,\varphi)$ is satisfied, for any $q\in \partial V_r(p)$ we have $\mathrm{FA}_1\big(M,g,q,R-r,\varphi(\varphi^{-1}(q)+\cdot)\big)$. Now, $r-\rho=\frac{1-\rho_0}{R_0-1}(R-r)$ with $\frac{1-\rho_0}{R_0-1}<1$, thus we can apply Proposition \ref{intégrale dehors}, giving for any $\tau \in [0,t]$:
	\[\int_{M\setminus V_{r-\rho}(q)} H(\tau,q,z) \mathrm dV(z)\le C_n \mathrm{e}^{-c (r-\rho)^2/\tau}\le C_n\mathrm{e}^{-cr^2/t}.\]
	where $c$ has been taken smaller in the last inequality. The integral bound on $|K(q,\cdot)-K_{V_r(p)}(q,\cdot)|$ follows by performing an integration on $\mathbf R_+$ against $\mathrm dt/t^{1+s/2}$.
\end{proof}

\subsection{Nonlocal mean curvature of subsets of arbitrary manifolds} \label{subsec: NMC bis}

We prove the general version of Proposition \ref{prop: NMC bien définie}, then show the equivalence of the two definitions of NMC boundedness. Recall that $\mathrm H_s[E]$ has been defined in \eqref{eq: def NMC}.

\begin{prop}\label{prop: NMC bien définie bis}Let $(M,g)$ be a smooth Riemannian manifold and consider a measurable subset $E\subset M$. Assume that $E$ has an interior (resp. exterior) tangent ball at $p\in\partial E$. Then $\mathrm H_s[E](p)$ is well-defined, with value in $(-\infty,+\infty]$ (resp. in $[-\infty,+\infty)$).\end{prop}
\begin{proof} Assume that $E$ has an interior tangent ball at $p\in \partial E$. We translate the problem to $\mathbf R^n$ in order to invoke Proposition \ref{prop: NMC bien définie}. Fix $r>0$ such that $\mathrm{FA}_1(M,g,p,r,\varphi)$ holds, for some smooth diffeomorphism $\varphi$. By Propositions \ref{intégrale dehors} and \ref{corimportant}, $\int_{M\backslash V_\varepsilon(p)} (\chi_E-\chi_{\mathcal CE})(q)K_g(p,q)\mathrm dV_g(q)$ converges as $\varepsilon\to 0$ if and only if 
	\[\int_{V_r(p)\backslash V_\varepsilon(p)} (\chi_E-\chi_{\mathcal CE})(q)K_{V_r(p),g}(p,q)\mathrm dV_g(q)\]
	converges. Let $F\subset \mathbf R^n$ be defined by $F=\varphi^{-1}(E\cap V_R(p))$, and let $h$ be a smooth, compact perturbation of the Euclidean metric, with $h=\varphi^*g$ in $B_r(0)$. Then the integral above equals 
	\begin{equation}\label{preuve NMC var} \int_{B_r(0)\backslash B_\varepsilon(0)}(\chi_F-\chi_{\mathcal CF})(x) K_{{B_r(0)},h}(0,x)\mathrm dV_h(x).\end{equation}
	
	The convergence of \eqref{preuve NMC var} as $\varepsilon\to 0$ is equivalent to that of $\int_{B_r(0)\backslash B_\varepsilon(0)} (\chi_F-\chi_{\mathcal CF})(x) K_{h}(0,x)\mathrm dV_h(x)$, which is clear by Proposition \ref{prop: NMC bien définie}, since $F$ has an interior tangent ball at $0$.
	
\end{proof}
\begin{rem}One may argue that the shrinking sets $V_\varepsilon(p)$ are not metric balls centered at $p$. However, recalling Remark \ref{Remarque shrinking balls}, one can show that the principal value does not depend on the choice of $\varphi$ in the proof, thus we may take $\varphi$ to be the exponential map $B_r(0)\to B_r^g(p)$ for $r$ small enough, in which case we have indeed $V_\varepsilon(p)=B_\varepsilon^g(p)$.\end{rem}

We now turn to the proof of Proposition \ref{prop: equiv def NMC} asserting the equivalence of the two definitions of NMC boundedness. 
\begin{proof}[Proof of Proposition \ref{prop: equiv def NMC}] One implication is clear because $F\subset E$ implies $\chi_F-\chi_{\mathcal CF}\le \chi_E-\chi_{\mathcal CE}$ thus $\mathrm{H}_s[F](y)\le \mathrm{H}_s[E](y)$ whenever $y\in \partial F\cap \partial E$.

	Let us now prove the converse implication. Assume $E$ has an interior tangent ball at $p\in \partial E$. This implies by Proposition \ref{prop: NMC bien définie bis} that $E$ has well-defined NMC at $p$ and $\mathrm{H}_s[E](p)\in (-\infty,+\infty]$. By very definition, we can find a diffeomorphism $\psi:B_1(0)\to V\subset M$ with $\psi(0)=p$ and $V^+:=\psi(B_1^+)\subset E$. For $\varepsilon>0$, let $\psi_\varepsilon:B_\varepsilon(0)\to V_\varepsilon$ be the restriction of $\psi$ to $B_\varepsilon(0)$. Letting $F_\varepsilon:=V_\varepsilon^+\cup (E\backslash V_\varepsilon)$, by assumption, we have $\mathrm{H}_s[F_\varepsilon]\le C_0$. Now,
	\[\mathrm{H}_s[F_\varepsilon](p)=\int_{V_\varepsilon(p)} (\chi_{V^+}-\chi_{\mathcal CV^+})(q)K(p,q)\mathrm dV(q)+\int_{M\backslash V_\varepsilon(p)} (\chi_E-\chi_{\mathcal C E})(q)K(p,q)\mathrm dV(q).\]
	Since $\partial V^+$ is smooth in a neighborhood of $p$, the first term of the right-hand side goes to $0$ as $\varepsilon\to 0$, while the second term converges to $\mathrm{H_s}[E](p)\in (-\infty,+\infty]$. It immediately follows that $\mathrm{H}_s[E](p)\le C_0$, as we wished.
\end{proof}

\subsection{Proof of Proposition \ref{regarderlocalement}}

Here, we combine the results of the previous subsections to prove Proposition \ref{regarderlocalement}. Recall that $\eta<\rho<r<R$, and $2r-\rho<R$, where $\eta,\rho,R$ are multiples of $r$ by explicit constants, that will be fixed. The proof is essentially the same as that of Proposition \ref{prop: NMC bien définie bis}, though here we have to be more careful in order to provide quantitative estimates.

\begin{proof} Assume $F$ has an interior tangent ball at a point $y\in B_\eta(0)$. Then, $E$ has an interior tangent ball at $w=\varphi(y)\in V_\eta(p)$, and the NMC boundedness assumption on $E$ reads
	\[\mathrm{p.v.}\int_M (\chi_E-\chi_{\mathcal CE})(q)K_g(w,q) \mathrm dV_g(q)\le C_0 r^{-s}\]
	By the tail estimate from Proposition \ref{intégrale dehors} we infer
	\[\mathrm{p.v.}\int_{V_\rho(p)} (\chi_E-\chi_{\mathcal CE})(q)K_g(w,q) \mathrm dV_g(q)\le (C_0+C_{n,s})r^{-s}.\]
	Then Proposition \ref{corimportant} shows that we can replace $K_g$ by $K_{V_r(p),g}$, only adding an integrable term, giving
	\[\mathrm{p.v.}\int_{V_\rho(p)} (\chi_E-\chi_{\mathcal CE})(q)K_{V_r(p),g}(w,q) \mathrm dV_g(q)\le (C_0+C_{n,s})r^{-s},\]
	where the constant $C_{n,s}$ has increased. Since $K_{\varphi(B_r(0)),g}(t,w,q)=K_{B_r(0),\varphi^*g}(t,\varphi^{-1}(w),\varphi^{-1}(q))$, after a change of variables we obtain,
	\begin{equation}\label{alsoreads}\mathrm{p.v.}\int_{B_\rho(0)} (\chi_F-\chi_{\mathcal CF})(x) K_{B_r(0),\varphi^*g}(x, y)\mathrm dV_{\varphi^*g}(x)\le (C_0+C_{n,s}) r^{-s}.\end{equation}
	Take a cutoff function $\chi$ on $\mathbf R^n$ such that $\chi\equiv 1$ in $B_r(0)$ and $\chi\equiv 0$ outside $B_{R}(0)$, with $\|\nabla \chi\|_{L^{\infty}}\le \frac{C}{R_0-1} r^{-1}$.
	Define a metric $h$ on $\mathbf R^n$ by $h=\chi(\varphi^*g)+(1-\chi)\mathrm{Id}$. Then $h$ clearly checks the claimed properties, and \eqref{alsoreads} equivalently reads
	\[\mathrm{p.v.}\int_{B_\rho(0)} (\chi_F-\chi_{\mathcal CF})(x) K_{B_r(0),h}(x, y)\mathrm dV_{h}(x)\le (C_0+C_{n,s})r^{-s}.\]
	Now we apply Proposition \ref{corimportant} and the tail estimate for $K_{h}$ backwards to obtain
	\[\mathrm{p.v.}\int_{\mathbf R^n} (\chi_F-\chi_{\mathcal CF})(x) K_{h}(x,y)\mathrm dV_{h}(x)\le (C_0+C_{n,s})r^{-s},\]
	where again the constants may have changed. The reasoning is the same if $F$ has an exterior tangent ball at $y$.
\end{proof}







\printbibliography

@article{CRS,
	author = {Caffarelli, L. and Roquejoffre, J.-M. and Savin, O.},
	title = {Nonlocal minimal surfaces},
	journal = {Communications on Pure and Applied Mathematics},
	year = {2010},
	volume ={63},
	number={9},
	pages={1111-1144},
}

@misc{Lombardini,
      title={Fractional Perimeter and Nonlocal Minimal Surfaces}, 
      author={Luca Lombardini},
      year={2015},
    type ={Master's thesis},
      eprint={1508.06241},
      archivePrefix={arXiv},
      primaryClass={math.AP}
}

@inproceedings{Grigoryan,
  title        = {Estimates of heat kernels on Riemannian manifolds},
  author       = {Grigor'Yan, A.},
  year         = {1999},
  booktitle    = {Spectral Theory and Geometry. ICMS Instructional Conference, Edinburgh, 1998},
  publisher    = {Cambridge University Press},
  series       = {London Math. Soc. Lecture Notes},
  pages        = {140-225},
  editor      = {Brian Davies and Yuri Safarov},
volume = {273}

}

@book{Grigorbook,
  title={Heat Kernel and Analysis on Manifolds},
  author={Grigor'Yan, A.},
  series={AMS/IP studies in advanced mathematics},
  year={2009},
  publisher={American Mathematical Society},
volume ={47}
}

@misc{caselli2024fractional,
      title={Fractional Sobolev spaces on Riemannian manifolds}, 
      author={Michele Caselli and Enric Florit-Simon and Joaquim Serra},
      year={2024},
type={Preprint},
      eprint={2402.04076},
      archivePrefix={arXiv},
      primaryClass={math.AP}
}

@article{CaffaValdi,
	author = {Caffarelli, L. and Valdinoci, E.},
	title = {Regularity properties of nonlocal minimal surfaces via limiting arguments},
	journal = {Advances in Mathematics},
	date = {2013},
	pages={843-871},
	volume={248}
}

@misc{Cha+,
	title={Nonlocal approximation of minimal surfaces: optimal estimates from stability}, 
	author={Hardy Chan and Serena Dipierro and Joaquim Serra and Enrico Valdinoci},
	year={2023},
type={Preprint},
	eprint={2308.06328},
	archivePrefix={arXiv},
	primaryClass={math.DG}
}

@article{cabré,
	author = {Cabré, X.},
	title = {
		Calibrations and null-lagrangians for nonlocal perimeters and application to the viscosity theory},
	journal = {Annali di Matematica Pura ed Applicata},
	date = {2020},
	volume={199},
	pages={1979-1995}
}

@misc{SSC,
	title={Yau's conjecture for nonlocal minimal surfaces}, 
	author={Michele Caselli and Enric Florit-Simon and Joaquim Serra},
	year={2024},
type={Preprint},
	eprint={2306.07100},
	archivePrefix={arXiv},
	primaryClass={math.DG}
}

@article{grigoribero,
	author = {Grigor'Yan, A.},
	title = {Heat  kernel upper bounds on a complete non-compact manifold},
	journal = {Revista Matemativa Iberoamericana},
	date = {1994},
	volume ={10},
	number={2},
	pages={395-452}
}

@article{BGS,
	author = {Banica, V. and Gonzáles, M. and Sáez, M.},
	title = {Some constructions for the fractional Laplacian on noncompact manifolds},
	journal = {Revista Matematica Iberoamericana},
	date = {2014},
	volume = {31},
	number={2},
	pages  ={681-712}
}

@Article{Dav,
	author={D{\'a}vila, J.},
	title={On an open question about functions of bounded variation},
	journal={Calculus of Variations and Partial Differential Equations},
	year={2002},
	volume={15},
	number={4},
	pages={519-527},
}

@article{Fig,
	author = {Barrios, B. and Figalli, A. and Valdinoci, E.},
	title = {
		Bootstrap regularity for integro-differential operators and its application to nonlocal minimal surfaces},
	journal = {Annali della Scuola Normale Superiore di Pisa},
	date = {2014},
	volume = {13},
	pages = {609-639},
}

@article{Song23,
	author={Song, A.},
	title ={Existence of infinitely many minimal hypersurfaces in closed manifolds},
	journal = {Annals of Mathematics},
	year ={2023},
	volume ={197},
	number={3},
	pages={859-895},
}

\end{document}